\documentclass[11pt]{article}
\usepackage[hidelinks]{hyperref}
\usepackage[sort&compress]{natbib}
\usepackage{mathtools}
\usepackage{blindtext}

\usepackage{natbib}
\setcitestyle{numbers,square}

\usepackage{amssymb,amsmath,amsfonts,amsthm,array,bm,bbm,color}%
\usepackage{verbatim}
\usepackage{graphicx}
\providecommand{\U}[1]{\protect\rule{.1in}{.1in}}
\usepackage{setspace}

\numberwithin{equation}{section}
\numberwithin{equation}{section}
\newtheorem{thm}{Theorem}[section]
\newtheorem{prop}[thm]{Proposition}
\newtheorem{lem}[thm]{Lemma}
\newtheorem{cor}[thm]{Corollary}
\newtheorem{rmk}[thm]{Remark}
\newtheorem{define}[thm]{Definition}
\newtheorem*{lem*}{Lemma}

\def\<{\langle}
\def\>{\rangle}
\def\d{{\rm d}}

\def\div{{\rm div \,}}

\def\E{\mathbb{E}}

\def\P{\mathbb{P}}
\def\R{\mathbb{R}}

\def\eps{\varepsilon}

\title{Well-posedness of stochastic mSQG equations with\\ Kraichnan noise and $L^p$ data}

\author{Shuaijie Jiao\footnotemark[1] \quad Dejun Luo\footnotemark[2]}

\begin{document}
	\begin{sloppypar}

		\maketitle
		\vspace{-20pt}
		\renewcommand{\thefootnote}{\fnsymbol{footnote}}
		\footnotetext[1]{Email: jiaoshuaijie@amss.ac.cn. School of Mathematical Sciences, University of Chinese Academy of Sciences, Beijing 100049, China, and Academy of Mathematics and Systems Science, Chinese Academy of Sciences, Beijing 100190, China}
		
		\footnotetext[2]{Email: luodj@amss.ac.cn. Key Laboratory of RCSDS, Academy of Mathematics and Systems Science, Chinese Academy of Sciences, Beijing 100190, China, and School of Mathematical Sciences, University of Chinese Academy of Sciences, Beijing 100049, China}
		
		\begin{abstract}
			We consider stochastic mSQG (modified Surface Quasi-Geostrophic) equations with multiplicative transport noise of Kraichnan type, and $L^p$-initial conditions. Inspired by the recent work of Coghi and Maurelli [arXiv:2308.03216], we show weak existence and pathwise uniqueness of solutions to the equations for suitable choices of parameters in the nonlinearity, the noise and the integrability of initial data.
		\end{abstract}
		
		
\section{Introduction}\label{sec-intro}

We consider the stochastic modified Surface Quasi-Geostrophic (mSQG) equation driven by multiplicative noise of transport type:
\begin{equation}\label{SmSQG}
	\left\{\aligned
	& \d \theta + u\cdot \nabla\theta\,\d t+ \circ\d W\cdot \nabla\theta =0,\\
	& u= \nabla^\perp (-\Delta)^{-\beta/2} \theta,
	\endaligned \right.
\end{equation}
where $\nabla^\perp= (\partial_2, -\partial_1)$, $\beta\in (1,2)$, $\circ\d$ stands for the Stratonovich stochastic differential, and $W=W(t,x)$ is a space-time noise which is white in time, colored and divergence free in space.  The existence of weak solutions to \eqref{SmSQG} with $L^p$-initial data is relatively classical for suitable choices of $\beta$ and $p>1$. In this paper, we are mainly interested in the uniqueness of solutions to \eqref{SmSQG} in the $L^p$-setting. Motivated by Coghi and Maurelli's work \cite{CogMau} on the stochastic 2D Euler equations, we will show that \eqref{SmSQG} enjoys pathwise uniquenss of weak solutions in suitable spaces, provided that the random perturbation $W(t,x)$ is the famous Kraichnan noise (see \cite{Kraichnan, Kra2}), whose covariance matrix $Q(x-y)= \E(W(1,x) \otimes W(1,y))$ is characterized by its Fourier transform:
\begin{equation}\label{Kraichnan-covariance}
	\hat Q(\xi)= (1+|\xi|^2)^{-1-\alpha} \bigg(I_2 - \frac{\xi\otimes \xi}{|\xi|^2} \bigg), \quad \xi\in \R^2,
\end{equation}
where $\alpha\in (0,1)$ and $I_2$ is the $2\times 2$ identity matrix. We remark that, under very general conditions on $Q$ (see e.g. \cite[Section 2.1]{GalLuo23}), the noise admits the decomposition $W(t,x)= \sum_k \sigma_k(x) B^k_t$, where $\{\sigma_k \}_{k\ge 1}$ are divergence free vector fields and $\{B^k \}_{k\ge 1}$ are independent Brownian motions; thus the first equation in \eqref{SmSQG} can be written more precisely as
$$\d \theta + u\cdot \nabla\theta\,\d t+ \sum_k \sigma_k\cdot \nabla\theta\circ\d B^k_t =0. $$

\subsection{Motivations}

The deterministic mSQG equation
\begin{equation}\label{DmSQG}
	\left\{\aligned
	& \partial_t \theta + u\cdot \nabla\theta =0,\\
	& u= \nabla^\perp (-\Delta)^{-\beta/2} \theta
	\endaligned \right.
\end{equation}
serves as a bridge linking the classical 2D Euler equation in vorticity form ($\beta=2$) and the SQG equation ($\beta=1$), the latter being used in meteorology and oceanography to model the temperature $\theta$ in a rapidly rotating stratified fluid with uniform potential vorticity, cf. \cite{HPGS95, Ped87}. In the influential work \cite{CMT94}, Constantin et al. pointed out the structural similarities between the SQG equation and the 3D Euler equation; more precisely, $\nabla^\perp \theta$ satisfies an equation which looks similar to the 3D Euler equation in vorticity form. Let us briefly recall a few well-posedness results for the SQG equation, i.e. \eqref{DmSQG} with $\beta=1$. Resnick proved in his PhD thesis \cite{Res95} the existence of weak solutions to the SQG equation with $L^2$-initial data; the result was extended by Marchand \cite{Mar08} to the $L^p$-setting with $p>4/3$, leaving open the uniqueness of solutions. Using the techniques of convex integration, Buckmaster et al. \cite{BSV19} proved the nonuniqueness of weak solutions to the SQG equation, even in the presence of fractional dissipation term.

The mSQG equation \eqref{DmSQG} was introduced in \cite{CFMR05} to approach the SQG equation, and therein the authors have shown evidence of formation of singularities in finite time. The existence of local strong solutions for smooth initial data in $C^r\, (r>1)$ is known, but the global existence is open. Thanks to the conservation form of \eqref{DmSQG}, the family of equations preserve $L^p$-norm of solutions, at least for smooth solutions. Chae et al. \cite{ChCoWu} introduced more general 2D inviscid models which include the mSQG equation \eqref{DmSQG} and the log-Euler equation. The local existence of vortex patch problem related to \eqref{DmSQG} was studied by Gancedo \cite{Gan08} and Kiselev et al. \cite{KYZ17}, while the validity of point vortex description for mSQG equation was established by Geldhauser and Romito \cite{GelRom20}. In the recent paper \cite{LS21}, the second author and Saal considered the point vortex system associated to the mSQG equation \eqref{DmSQG}, and proved that a certain non-degenerate and space-dependent noise prevents the collapse of vortex system, extending a previous result in \cite{FGP11} on the point vortex system of 2D Euler equations. Flandoli and Saal \cite{FS19} proved the existence of stationary white noise solutions to \eqref{DmSQG} for $\beta\in (1,2)$; Luo and Zhu \cite{LZ21} obtained similar results for stochastic mSQG equations with transport noise and showed, under a suitable scaling of the noise, that the white noise solutions converge weakly to the unique stationary solution of the dissipative mSQG equation driven by space-time white noise; see \cite{FlLu20} for related results for 2D Euler equations.  We refer to the introduction of \cite{FS19} for some other references related to the mSQG equation \eqref{DmSQG}.

The 2D Euler equation in vorticity form (i.e. \eqref{DmSQG} with $\beta=2$)
\begin{equation}\label{2D-Euler}
	\left\{\aligned
	& \partial_t \theta + u\cdot \nabla\theta =0,\\
	& u= \nabla^\perp (-\Delta)^{-1} \theta
	\endaligned \right.
\end{equation}
is a fundamental model in fluid dynamics, where $u$ and $\theta$ now represent the fluid velocity and vorticity, respectively. The classical result of Yudovich \cite{Yud63} asserts that if $\theta_0\in (L^1\cap L^\infty)(\R^2)$, then \eqref{2D-Euler} admits a unique weak solution in $L^\infty([0,T];L^1\cap L^\infty)$. Since then, it has been a challenging open problem in fluid dynamics to show the uniqueness of weak solutions to \eqref{2D-Euler} for $(L^1\cap L^p)$-initial data with $p\in (1,\infty)$. In recent years, however, a number of nonuniqueness results for \eqref{2D-Euler} appeared. Bressan et al. \cite{BreMur, BreShe} proposed a (numerically assisted) scheme for showing the existence of nonunique weak solutions to \eqref{2D-Euler}. By adding a carefully designed forcing term $f$ to \eqref{2D-Euler}, Vishik \cite{Vishik1, Vishik2} constructed nonunique solutions with null initial condition; the construction was revisited by De Lellis and his group in \cite{ABCDGJK} where the authors improved many of the arguments. Albritton et al. \cite{AlBrCo} adapted the ideas in \cite{ABCDGJK} to construct nonunique Leray-Hopf solutions of 3D Navier-Stokes equations with an external force.

On the other hand, the growing theory of regularization by noise demonstrates that the addition of suitable noises can improve the solution theory of many systems, see \cite{FlaBook, Gess} for surveys of some early results and the introductions of \cite{CogMau, GalLuo23} for more references. In the case of fluid dynamics equations like \eqref{2D-Euler}, it is by now generally accepted that the transport noise in Stratonovich form is a physically well motivated random perturbation (cf. \cite{FP22, Holm2015}). Thus, the problem is to search for appropriate space-time noises $W(t,x)= \sum_k \sigma_k(x) B^k_t$ such that the solutions to
\begin{equation}\label{stoch-2D-Euler}
	\left\{\aligned
	& \d \theta + u\cdot \nabla\theta\, \d t + \sum_k \sigma_k\cdot \nabla\theta\circ\d B^k_t =0,\\
	& u= \nabla^\perp (-\Delta)^{-1} \theta
	\endaligned \right.
\end{equation}
are unique in a certain sense. Flandoli \cite{Fla15} proposed a simplified problem in which the nonlinearity is slightly regularized, namely, the relation between $u$ and $\theta$ is replaced by
$$u= \nabla^\perp (-\Delta)^{-1-\eps} \theta,$$
where $\eps>0$ is a given small number. He also discussed possible approaches for solving the problem; a key idea, originally due to P. Malliavin, is to use the Girsanov transform which formally removes the nonlinearity and gives rise to a stochastic linear transport equation. This strategy was rigorously implemented by Galeati and the second author \cite{GalLuo23}, yielding the uniqueness in law of weak solutions to the regularized stochastic 2D Euler equation with any $\eps>0$ and appropriately chosen noise.

In the recent work \cite{CogMau}, Coghi and Maurelli studied the true 2D Euler equation perturbed by Kraichnan transport noise, namely, the noise $W(t,x)$ in \eqref{stoch-2D-Euler} has a covariance matrix function $Q$ characterized by \eqref{Kraichnan-covariance}. Thanks to the noise, for initial data $\theta_0$ in the homogeneous Sobolev space $\dot H^{-1}$, they are able to show existence of weak solutions to \eqref{stoch-2D-Euler} in the space $L^\infty\big([0,T], L^2(\Omega, \dot H^{-1})\big) \cap L^2([0,T]\times \Omega, H^{-\alpha})$, $\Omega$ being the probability space where the noise $W$ lives (possibly different from the original one), and $\alpha\in (0,1)$ is the parameter in \eqref{Kraichnan-covariance}. Furthermore, for initial data in $L^1\cap L^p\cap \dot H^{-1}$ with suitable choices of $p>1$ and $\alpha\in (0,1)$, Coghi and Maurelli successfully proved the pathwise uniqueness of solutions to \eqref{stoch-2D-Euler}, thus obtaining existence of probabilistically strong solutions. Let us briefly discuss the origin of the additional $L^2([0,T]\times \Omega, H^{-\alpha})$-regularity, which is the key ingredient for passing to the limit in the nonlinear terms of approximate equations. Let $(\theta, u)$ be a solution to \eqref{stoch-2D-Euler}. Transforming \eqref{stoch-2D-Euler} in It\^o equations and formally applying the It\^o formula to $\|\theta \|_{\dot H^{-1}}^2= \<\theta, G\ast \theta\>$, where $G$ is the Green function on $\R^2$, one obtains
$$\d \|\theta \|_{\dot H^{-1}}^2 = \d M_t - 2\<G\ast \theta, u \cdot \nabla\theta\>\,\d t + \big\< {\rm tr}[(Q(0)- Q) D^2 G]\ast \theta, \theta\big\>\,\d t, $$
where $\d M_t$ is the martingale part, ${\rm tr}[\cdot]$ means trace of matrices and the last quantity comes from the noise term by integration by parts. It is easy to see that $\<G\ast \theta, u \cdot \nabla\theta\>=0$ after integrating by parts. The key estimate in \cite{CogMau} is that
\begin{equation}\label{key-estimate}
	\big\< {\rm tr}[(Q(0)- Q) D^2 G]\ast \theta, \theta\big\> \le -c \|\theta \|_{H^{-\alpha}}^2 + C \|\theta \|_{\dot H^{-1}}^2
\end{equation}
for some positive constants $c, C>0$. Inserting this estimate into the above identity, taking expectation and applying Gronwall's inequality lead to the desired regularity estimates of the solutions. We emphasize that, in the proof of pathwise uniqueness, the negative quantity in \eqref{key-estimate} is also the key ingredient to cancel the terms arising from the nonlineaity.

Inspired by Coghi and Maurelli's work, we want to study the stochastic mSQG equation \eqref{SmSQG} with Kraichnan transport noise. Our main purpose is to show the pathwise uniqueness of solutions to \eqref{SmSQG} for suitable choices of parameters $\alpha\in (0,1)$, $\beta\in (1,2)$, and $p>1$ which stands for the integrability of initial data. For this purpose, we will first provide an existence result in the next subsection. Before going to the details, we mention a major difference between the nonlinearities of equations \eqref{SmSQG} and \eqref{stoch-2D-Euler}. Indeed, we have $u= \nabla^\perp (-\Delta)^{-1} \theta$ in \eqref{stoch-2D-Euler}, and one easily deduces that
$$u\cdot \nabla\theta = \nabla^\perp \cdot \div(u\otimes u). $$
This identity simplifies the estimates of $u\cdot \nabla\theta$ in many cases, see e.g. the term $\tilde S_1$ on page 37 of \cite{CogMau}. However, we no longer have such simple relation for the nonlinearity in equation \eqref{SmSQG}; this fact has several consequences, for instance, to show the existence of solutions, we need to assume that the initial data is in $L^p$ for some suitable $p>1$.

\subsection{Main results}

First, we give the It\^o formulation of \eqref{SmSQG} with Kraichnan noise $W(t,x)= \sum_k \sigma_k(x) B^k_t$. Thanks to the explicit expression \eqref{Kraichnan-covariance}, one can show that (see \cite[Section 2.3]{CogMau})
$$Q(0)= \E[W(1,x)\otimes W(1,x)] = \sum_k \sigma_k(x) \otimes \sigma_k(x) = \frac\pi{2\alpha} I_2, \quad x\in \R^2. $$
From this, some simple computations lead to the It\^o equation of \eqref{SmSQG}:
\begin{equation}\label{mSQG}
	\left\{\aligned
	& \d \theta + u\cdot \nabla\theta\,\d t+ \sum_k \sigma_k\cdot \nabla\theta\,\d B^k_t = \frac\pi{4\alpha} \Delta\theta\,\d t,\\
	& u= \nabla^\perp (-\Delta)^{-\beta/2} \theta.
	\endaligned \right.
\end{equation}
In the sequel we always work with this formulation. We will often write $u= K_\beta\ast \theta$ where $K_\beta$ is the kernel corresponding to $\nabla^\perp (-\Delta)^{-\beta/2}$; note that $K_\beta$ reduces to the classical Biot-Savart kernel if $\beta=2$.

\begin{define}\label{def-sol}
	Let $\alpha\in (0,1), \beta\in (1,2)$ and $p>1$ be given. A weak solution in $\dot H^{-\frac\beta 2} \cap L^p$ to \eqref{mSQG} is a collection of objects $\big(\Omega, \mathcal F, (\mathcal F_t)_t, \P, (B^k)_k, \theta \big)$ where $(\Omega, \mathcal F, (\mathcal F_t)_t, \P)$ is a filtered probability space satisfying the usual conditions, $(B^k)_k$ are independent real Brownian motions, $\theta:[0,T]\times \Omega\to \big(\dot H^{-\frac\beta 2} \cap L^p \big)(\R^2)$ is an $(\mathcal F_t)_t$-progressively measurable process satisfying
	\begin{equation}\label{path-bd}
		\P\text{-a.s.}, \quad \theta\in L^\infty \big([0,T], \dot H^{-\frac\beta 2}\big) \cap C\big([0,T], H^{-4} \big), \quad \sup_{t\in [0,T]} \|\theta_t \|_{L^p} \le \|\theta_0 \|_{L^p},
	\end{equation}
	and it holds
	\begin{equation}\label{weak-solution}
		\theta_t= \theta_0 -\int_0^t\div(u_r\theta_r)\,\d r - \sum_k \int_0^t \div(\sigma_k \theta_r)\, \d B^k_r + \frac\pi{4\alpha} \int_0^t \Delta\theta_r\,\d r, \quad \forall\, t\in [0,T]
	\end{equation}
	as an equality in $H^{-4}$.
\end{define}

The following result gives the existence of weak solutions to \eqref{mSQG}.

\begin{thm}\label{exist}
	Let $\alpha\in (0,1)$, $\beta\in (1,2)$ satisfy $1-\frac{\beta}{2}<\alpha< \frac{\beta}{2}$, and $\max \big\{\frac{2}{1+\beta/2 -\alpha}, \frac4{\beta+1} \big\} <p \le 2$, then for any $\theta_0\in \dot H^{-\frac\beta 2} \cap L^p$, the stochastic mSQG equation \eqref{mSQG} admits a weak solution in the sense of Definition \ref{def-sol}, satisfying in addition
	%$$\theta\in L^\infty\big([0,T], L^2(\Omega, \dot H^{-\frac\beta2}) \big) \cap L^2\big([0,T]\times \Omega, H^{-\frac\beta2 +1-\alpha}\big) $$
	\begin{equation}\label{exist-bound}
		\sup_{t\in [0,T]} \E\Big[\|\theta _t\|_{\dot H^{-\frac\beta2}}^2 \Big] + \int_0^T \E\Big[\|\theta _t\|_{H^{-\frac\beta2 +1-\alpha}}^2 \Big]\,\d t<+\infty.
	\end{equation}
\end{thm}

The proof of this result follows the classical compactness argument: we approximate the initial condition, the nonlinearity and the noise with smooth objects, prove uniform estimates on the solutions to approximate equations, which imply compactness of their laws in suitable spaces, then we apply the Skorohod representation theorem to get almost sure convergent sequence on some new probability space, finally we show that the limit is a weak solution to the stochastic mSQG equation. Let us briefly explain the roles of our conditions.

\begin{rmk}\label{rem-existence}
	\begin{itemize}
		\item[(a)] Note that $\beta >2\alpha$ implies $\frac{2}{1+\beta/2 -\alpha}<2$, and $\beta>1$ implies $\frac4{\beta+1}<2$, so the range of $p$ is not empty. Next, the condition $2\alpha+ \beta>2$ will be used in Lemma \ref{extrareg} below to yield the $ L^2\big([0,T]\times \Omega, H^{-\frac\beta2 +1-\alpha}\big)$-regularity of solutions.
		
		\item[(b)] Uniform estimate like \eqref{exist-bound} on the approximate solutions $\{\theta^\delta\}_{\delta>0}$ implies that $\{ u^\delta= K_\beta\ast \theta^\delta \}_{\delta>0}$ are bounded in $ L^2\big([0,T]\times \Omega, H^{\frac\beta2-\alpha}\big) $, and the condition $\beta >2\alpha$ ensures that the space $H^{\frac\beta2-\alpha}$ is compactly embedded in $L^2_{loc}$; these facts play important roles in passing to the limit of nonlinear terms in approximate equations.
		
		\item[(c)] The condition $p> \frac4{\beta+1}$ will be used to establish time continuity estimate of the nonlinear parts of approximate equations, see the beginning of the proof of Lemma \ref{time-continuity}. On the other hand, the condition $p> \frac{2}{1+\beta/2 -\alpha}$ ensures the convergence of $A_{22}$ which comes from the nonlinearities, see \eqref{A22} in \textbf{Step 3} of the proof of Theorem \ref{exist}.
	\end{itemize}
\end{rmk}

Now we state the main theorem of our paper which implies pathwise uniqueness of solutions to \eqref{mSQG} for suitable choices of parameters. We write $a\wedge b =\min\{a, b\}$ for $a,b\in \R$.

\begin{thm}\label{unique}
	Let $\alpha\in (0,\frac{1}{2})$, $\beta\in (\frac{3}{2},2)$ satisfy $\frac2{p\wedge 2} -\frac\beta2  <\alpha < \frac\beta2 -\frac1{p\wedge 2}$ for some $p> \frac3{\beta}$. Assume $\theta_0\in L^1 \cap L^p$, then the following assertions hold.
	\begin{itemize}
		\item[(1)] There exists a weak solution in the sense of Definition \ref{def-sol} to the stochastic mSQG equation \eqref{mSQG}, satisfying \eqref{exist-bound} and
		\begin{equation}\label{path-bd2}
			\P\mbox{-a.s.}, \quad \sup_{t\in[0,T]} (\|\theta_t \|_{L^1}+ \|\theta_t \|_{L^p}) \le \|\theta_0 \|_{L^1}+ \|\theta_0 \|_{L^p} .
		\end{equation}
		\item[(2)] Pathwise uniqueness holds for weak solutions of \eqref{mSQG} in the space
		$$\mathcal X:= L^\infty\big(\Omega\times [0,T], L^1 \cap L^p \big) .$$
		%\cap L^2\big(\Omega\times [0,T], H^{-\frac\beta2 +1-\alpha}\big) . $$
		More precisely, if $\theta^1,\, \theta^2\in \mathcal X$ are two weak solutions to \eqref{mSQG} in the sense of Definition \ref{def-sol}, with the same initial condition $\theta_0$ and the same sequence of Brownian motions $(B^k)_k$ on the common filtered probability space $(\Omega, \mathcal F, (\mathcal F_t)_t, \P)$, then $\P$-a.s. for all $t\in [0,T]$, $\theta^1_t= \theta^2_t$.
	\end{itemize}
\end{thm}

\begin{rmk}\label{condition-unique}
	The conditions $\alpha\in (0,\frac{1}{2})$ and $\beta\in (\frac{3}{2},2)$ imply $\frac{3}{\beta}>\frac{4}{\beta+1}>\frac{2}{1+\beta/2 -\alpha}$.
	Recall the Sobolev embedding in dimension 2: $L^q \subset \dot H^{1-\frac2q}$ for any $1< q\le 2$, see e.g. \cite[Corollary 1.39]{Bahouri2011FourierAA}. By the restriction $p>\frac{3}{\beta}$, one easily sees that $L^1 \cap L^p \subset \dot H^{-\frac\beta2}$ and the class $\mathcal X$ is contained in $L^\infty\big([0,T], L^2(\Omega, \dot H^{-\frac\beta2}) \big)$, hence all conditions in Theorem \ref{exist} are satisfied here. In the same way, the condition $\alpha>\frac2p -\frac\beta2$, or equivalently $p> \frac{2}{\beta/2 +\alpha}$, implies that the class $\mathcal X$ is contained in $L^2\big(\Omega\times [0,T], H^{-\frac\beta2 +1-\alpha}\big)$.
\end{rmk}

\begin{rmk}\label{rem-uniqueness-2}
    The results in Theorem \ref{unique} can be understood as follows: for every fixed $\beta\in (\frac{3}{2},2)$, assume that the initial datum $\theta_0$ belongs to the space $L^1\cap L^p$ with $p>\frac{3}{\beta}$, then we can find a special multiplicative transport noise, namely the Kraichnan noise with some suitable index $\alpha$, such that the 2D stochastic mSQG equation \eqref{mSQG} admits a pathwise unique solution. Such a uniqueness result is not known for the deterministic mSQG equation \eqref{DmSQG}, thus the addition of noise improves the solution theory for this equation.
	%For the case $1<p\leq2$ in Theorem \ref{unique}, the restrictions on $\alpha$ and $\beta$ reduce to $\alpha\in (0,1)$, $\beta\in \big(\frac3p,2 \big),\, 2-\beta<\alpha<\frac{\beta}{2}$ and $ \frac{2}{p}-\frac{\beta}{2} <\alpha < \frac{\beta}{2}-\frac{1}{p}$ (the last condition implies $\alpha<1+\frac{\beta}{2}-\frac{2}{p}$ and thus $p> \frac2{1+\beta/2 -\alpha}$). Note that $\beta > \frac3p$ implies $\frac{2}{p}-\frac{\beta}{2} < \frac{\beta}{2}-\frac{1}{p}$, hence we need $2-\beta< \frac{\beta}{2}-\frac{1}{p}$ which is equivalent to $\beta> \frac43 + \frac2{3p}$. These considerations lead to $\max\big\{\frac3p, \frac{4}{3}+\frac{2}{3p} \big\} <\beta <2$; in particular, if $p=2$, we need $\beta>\frac53$.
	%then the restrictions on $\alpha$ and $\beta$ reduce to $\alpha\in (0,1)$, $\beta\in (\frac32, 2)$, $\beta> \max\{2\alpha, 2-\alpha\}$ and $1-\frac\beta 2 < \alpha < \frac\beta2 -\frac12$. Note that $\beta> 2-\alpha$ implies $1-\frac\beta 2 < \alpha$, so the last two requirements can be combined as $\beta > \max\{2-\alpha, 1+ 2 \alpha \}$, which is equivalent to $2-\beta < \alpha < \frac{\beta-1}2$. For this purpose, we need $\beta> \frac53$.
	
	%To sum up, for $\beta\in (\frac53, 2)$, we can find a Kraichnan noise $W= \sum_k \sigma_k B^k$ with parameter $\alpha\in (0,\frac12)$, such that the stochastic mSQG equation \eqref{mSQG} admits a pathwise unique solution for any initial data $\theta_0 \in L^1\cap L^2$. Such a uniqueness result is not known for the deterministic mSQG equation \eqref{DmSQG}, thus the addition of noise improves the solution theory for this equation.
\end{rmk}

Almost at the same time of our work, Bagnara, Galeati and Maurelli \cite{BGM24} studied the same regularization problem. They assume $p\ge 2$ and are able to treat the full range of $\beta\in (1,2)$ (note that $\beta$ here is related to the parameter in their paper by the relation $\beta\mapsto 2-\beta$); for $\beta$ close to 1, the parameter $p$ has to be big enough.

We finish the short introduction with the structure of the paper. In Section \ref{sec-prelim} we make some preparations for later use, in particular, we introduce smooth approximations of the Kraichnan covariance function $Q$, the Riesz kernel $G_\beta$ and the initial data $\theta_0\in \dot H^{-\frac\beta2}$. Then we consider in Section \ref{sec-a-priori} the approximate equations and prove several a priori estimates on their solutions. Theorem \ref{exist} will be proved in Section \ref{weak-exist} by using the classical compactness approach. Finally, we prove Theorem \ref{unique} in Section \ref{sec-uniqueness} by following the method of Coghi and Maurelli \cite{CogMau}.

\section{Preliminaries}\label{sec-prelim}
In this section, we present essential prerequisites used in the sequel, including some properties of the covariance of the noise and the kernel related to the nonlinear part in the mSQG equation \eqref{mSQG}. Then we follow the ideas in \cite[Section 3]{CogMau} to smooth the irregular elements in \eqref{mSQG}.

To begin with, we introduce some notations that are used throughout the paper.
Let $(\Omega,\mathcal{F},(\mathcal{F}_t)_t,\P)$ be a filterd probability space which satisfies the usual conditions and $\mathbb{E}$ denotes the expectation.
%The element of the sample space $\Omega$ is denoted by $\omega$.
For $x = (x^1,x^2)\in \mathbb{R}^2$, we define $x^{\perp} = (x^2,-x^1)$ as the vector obtained by rotating $x$ clockwise 90 degrees and define $\langle x \rangle=(1+\lvert x \rvert^2)^{\frac{1}{2}}$. Write the open (resp. closed) ball of center $x$ and radius $R$ as $B_R(x)$ (resp. $\bar{B}_R(x)).$

Throughout this paper, the notation $\langle\cdot,\cdot\rangle$ stands both for the scalar product in a Hilbert space and the pairing between a space and its dual. For a tempered distribution $u\in \mathcal{S}', \hat{u}$ denotes its Fourier transform. For $s\in \mathbb{R}$, let $\dot{H}^s(\mathbb{R}^2)\,(\text{resp. }H^s(\mathbb{R}^2))$ denote the usual homogeneous (resp. inhomogeneous) Sobolev spaces, see for instance \cite[Chapter 1]{Bahouri2011FourierAA}.
For $\gamma\in (0,1)$ and a Banach space $B$, we use the notation
\begin{equation*}
	\Vert f \Vert_{C_t^{\gamma}(B)}
	= \sup\limits_{t\in [0,T]} \Vert f(t) \Vert_B
	+ \sup\limits_{0\leq s\neq t \leq T}
	\dfrac{\Vert f(t)-f(s) \Vert_B}{\lvert t-s \rvert^\gamma},
\end{equation*}
\noindent to denote the norm of $\gamma$-H{\"o}lder continuous functions.

The symbol $f\lesssim g$ for two functions $f$ and $g$ means that there exists a positive constant $C$ such that $f(x)\leq Cg(x)$ for all $x$. We use the symbol $f\asymp g$ to stand for $f\lesssim g$ and $g\lesssim f$.

%Now we recall the definition both of the homogeneous and inhomogeneous Sobolev spaces, see for instance \cite[Chapter 1]{Bahouri2011FourierAA}. For $\xi\in \mathbb{R}^2$, we use the usual notation $\langle \xi \rangle:=(1+\lvert \xi \rvert^2)^{\frac{1}{2}}$.

%\begin{define}
%	Let $s$ be a real number. %$d\geq 2$ be an integer.
%	\begin{itemize}
%		\item The homogeneous Sobolev space $\dot{H}^s(\mathbb{R}^2)$ is defined as
%		\[\dot{H}^s(\mathbb{R}^2)
%		=\left\{u\in \mathcal{S}'(\mathbb{R}^2)
%		:\Vert u \Vert_{\dot{H}^s(\mathbb{R}^2)}^2
%		:=\int_{\mathbb{R}^2}\lvert \xi \rvert^{2s}
%		\lvert \hat{u}(\xi) \rvert^2\,\mathrm{d}\xi < \infty \right\}.\]

%		\item The inhomogeneous Sobolev space $H^s(\mathbb{R}^2)$ is defined as
%		\[H^s(\mathbb{R}^2)
%		=\left\{u\in \mathcal{S}'(\mathbb{R}^2)
%		:\Vert u \Vert_{H^s(\mathbb{R}^2)}^2
%		:=\int_{\mathbb{R}^2}\langle \xi \rangle^{2s}
%		\lvert \hat{u}(\xi) \rvert^2\, \mathrm{d}\xi < \infty \right\}.\]
%	\end{itemize}
%\end{define}

%\begin{rmk}
%	The homogeneous Sobolev space $\dot{H}^s(\mathbb{R}^2)$ is a Hilbert space if and only if $s<1$, see \cite[Proposition 1.34]{Bahouri2011FourierAA}
%\end{rmk}

\subsection{Mollifying the covariance $Q$}

Take %the regular index
$0<\alpha<1$, recall that the Kraichnan covariance $Q:\mathbb{R}^2\rightarrow \mathbb{R}^{2\times2}$ is determined by its Fourier transform
\begin{equation*}
	\hat{Q}(\xi)=\langle \xi\rangle^{-(2+2\alpha)}\bigg(I_2-\dfrac{\xi\otimes \xi}{\lvert \xi\rvert^2 }\bigg),\quad \xi\in \mathbb{R}^2.
\end{equation*}
The following properties about the covariance matrix are known, see e.g. \cite[Section 2]{CogMau} and \cite[Section 10]{LeJan-Raim02}.
\begin{lem}
	There exists a sequence of divergence-free vector fields $\sigma_k\in H^{1+\alpha}(\mathbb{R}^2,\mathbb{R}^2),k\in \mathbb{N}$, such that
	\begin{equation*}
		Q(x,y)=Q(x-y)=\sum_{k}\sigma_k(x)\otimes\sigma_k(y),\quad x,y\in \mathbb{R}^2
	\end{equation*}
	(the series converge absolutely for $x$ and $y$) and
	\begin{equation}\label{Q-converge}
		\sup\limits_{x,y\in \mathbb{R}^2}
		\sum_{k}\vert \sigma_k(x) \vert \, \vert \sigma_k(y) \vert\leq
		\sup\limits_{x\in \mathbb{R}^2}
		\sum_{k}\vert \sigma_k(x) \vert^2
		< \infty.
	\end{equation}
\end{lem}

\begin{lem}\label{Q-structure}
	We have:
	\begin{equation*}
		Q(x)=B_L(\vert x \vert)\dfrac{x\otimes x}{\vert x \vert^2}+B_N(\vert x \vert)\bigg(I_2-\dfrac{x\otimes x}{\vert x \vert^2}\bigg),
	\end{equation*}
	with
	\begin{align*}
		B_L(R)&=\frac{\pi}{2\alpha}-\beta_LR^{2\alpha}-\mathrm{Rem}_{1-u^2}(R), \\
		B_N(R)&=\frac{\pi}{2\alpha}-\beta_NR^{2\alpha}-\mathrm{Rem}_{u^2}(R), \\
		\beta_N&=(1+2\alpha)\beta_L>\beta_L>0,
	\end{align*}
	where the remainders satisfy $\vert \mathrm{Rem}_{1-u^2}(R) \vert+\vert \mathrm{Rem}_{u^2}(R) \vert \lesssim R^2$ for all $R>0$. In particular, we have
	\begin{equation*}
		\vert Q(0)-Q(x) \vert \lesssim \vert x \vert^{2\alpha}\wedge1,\quad \forall x\in \mathbb{R}^2.
	\end{equation*}
\end{lem}

\begin{rmk}
	Lemma \ref{Q-structure} describes the behaviour of the covariance $Q$ near $x=0$, which will play a key role in the energy bounds of $\theta$.
\end{rmk}

Now we show how to mollify the covariance $Q$. For $\delta>0$, suppose $\rho^\delta$ is a function such that $\widehat{\rho^\delta}$ is real-valued, radial and smooth with $0\leq \widehat{\rho^\delta} \leq 1$ everywhere, $\widehat{\rho^\delta}(\xi)=1$ on $\vert \xi \vert\leq \frac{1}{\delta}$ and $\widehat{\rho^\delta}(\xi)=0$ on $\vert \xi \vert\geq \frac{2}{\delta}$, hence $\rho^{\delta}$ is rapidly decreasing. Take
\begin{equation*}
	\begin{split}
		\sigma_k^{\delta}&=\rho^{\delta} \ast \sigma_k,\\
		Q^{\delta,h}&=\rho^{\delta} \ast Q,\\
		Q^{\delta}&=\rho^{\delta}*\rho^{\delta} \ast Q.
	\end{split}
\end{equation*}
Note that
\begin{equation}  \label{Q-delta}
	\widehat{Q^{\delta}}(\xi)
	=\hat{Q}(\xi)\widehat{\rho^\delta}(\xi)^2
	=\langle \xi\rangle^{-(2+2\alpha)}\bigg(I_2-\dfrac{\xi\otimes \xi}{\lvert \xi\rvert^2 }\bigg)
	\widehat{\rho^\delta}(\xi)^2.
\end{equation}
Therefore, $Q^{\delta}\in \mathcal{S}(\mathbb{R}^2)$ satisfies $\Vert Q^{\delta} \Vert_{L^{\infty}}
\leq \Vert \widehat{Q^{\delta}} \Vert_{L^1}
\leq \Vert \hat{Q} \Vert_{L^1}$.

Moreover, it is easy to check   $Q^{\delta,h}(x-y)=\sum_{k}\sigma_k^{\delta}(x)\otimes \sigma_k(y)=\sum_{k}\sigma_k(x)\otimes \sigma_k^{\delta}(y)$ and $Q^{\delta}(x-y)=\sum_{k}\sigma_k^{\delta}(x)\otimes \sigma_k^{\delta}(y)$. Also, as $\delta\rightarrow0$,
\begin{equation}\label{eq-c-delta}
	Q^{\delta}(0)=\frac{1}{2}\bigg(\int_{\mathbb{R}^2}
	\langle \xi\rangle^{-(2+2\alpha)}\widehat{\rho^\delta}(\xi)^2\,\mathrm{d}\xi\bigg)I_2=:c_{\delta}I_2\rightarrow\dfrac{\pi}{2\alpha}I_2=Q(0).
\end{equation}

\subsection{Regularizing the kernel $G_{\beta}$}
Our task now is to give some descriptions of the kernel corresponding to the nonlinear operator $\nabla^\perp(-\Delta)^{-\frac{\beta}{2}}$ in mSQG \eqref{mSQG}.
We recall some fundamental facts about the Riesz potentials in $\mathbb{R}^2$. For $0<\beta<2$, define
\begin{equation}
	G_{\beta}(x)=\dfrac{1}{\gamma(\beta)}\vert x \vert^{-2+\beta}, \quad
	\gamma(\beta)=\dfrac{2^{\beta}\pi%^{\frac{d}{2}}
		\Gamma(\frac{\beta}{2})}{\Gamma(\frac{2-\beta}{2})}.
\end{equation}
It is known that $\hat{G}_{\beta}(\cdot)=(2\pi \vert \cdot \vert)^{-\beta}$, hence we have $(-\Delta)^{-\frac{\beta}{2}}f=G_{\beta}*f$ for all $f\in \mathcal{S}'(\mathbb{R}^2)$. For more details, see \cite[Chapter V, Section 1]{Stein1971}.
In this paper, we are mainly concerned with the case $1<\beta<2$. The derivative and second derivative of $G_\beta$ are
\begin{align*}
	\nabla G_{\beta}(x)&=-\frac{2-\beta}{\gamma(\beta)}
	\dfrac{x}{\vert x \vert ^{4-\beta}}, \\
	D^2 G_{\beta}(x)&=-\frac{2-\beta}{\gamma(\beta)}
	\dfrac{1}{\vert x \vert ^{4-\beta}}\bigg(I_2-(4-\beta)\dfrac{x\otimes x}{\vert x \rvert^2}\bigg),
\end{align*}
and the kernel $K_{\beta}$ is defined as the orthogonal derivative of $G_{\beta}$:
\begin{equation}
	K_{\beta}(x)=\nabla^\perp G_{\beta}(x)
	=-\frac{2-\beta}{\gamma(\beta)}
	\dfrac{x^\perp}{\vert x \vert ^{4-\beta}}, \quad x\neq 0.
\end{equation}
Note that when $\beta=2$, $K_2$ is the Biot-Savart kernel on $\mathbb{R}^2$, which is usually denoted by $K$.

%\subsection{Fractional heat kernel}
%Now we turn to consider the fractional heat kernel on $\mathbb{R}^d$.
It is straightforward to verify that $G_\beta$ is the Green function of the fractional Laplacian operator $(-\Delta)^{\frac{\beta}{2}}$ on $\mathbb{R}^2$, that is,
\begin{equation*}
	(-\Delta)^{\frac{\beta}{2}}G_\beta = \delta_0.
\end{equation*}
To obtain `good' approximation of the Riesz potentials $G_\beta$, we exploit the relation between Riesz potential and the corresponding fractional heat kernel. Indeed, suppose that $p(t,x)$ is the fundamental solution of the fractional heat equation:
\begin{equation}
	\left\{
	\begin{lgathered}
		\partial_t p + (-\Delta)^{\frac{\beta}{2}}p=0
		\quad\text{in }(0,\infty)\times \mathbb{R}^2,\\
		p(0,\cdot)=\delta_0.
	\end{lgathered}
	\right.  \label{fraction heat eqn}
\end{equation}
Applying the Fourier transform to \eqref{fraction heat eqn}, we get  $\hat{p}(t,\xi)=e^{-(2\pi \vert \xi \vert)^{\beta}t}$, thus we have the following relation:
\begin{equation}\label{G}
	G_{\beta}(x)=\int_{0}^{\infty}p(t,x)\,\mathrm{d}t.
\end{equation}
%Recall that for given $1<\beta<2,G_\beta$ is the Green function of the fractional Laplacian $(-\Delta)^{\frac{\beta}{2}}$, and $p(t,x)$ is the fractional heat kernel as defined in \eqref{fraction heat eqn}.
Now for $0<\delta<1$, we define the regularized kernels $G_\beta^\delta,\, K_\beta^\delta$ as follows:
\begin{equation}\label{G-approx}
	G_\beta^\delta(x)=\int_{\delta}^{1/\delta}p(t,x)\, \mathrm{d}t,\quad K_\beta^\delta=\nabla^{\perp}G_\beta^\delta,\quad x\in \mathbb{R}^2.
\end{equation}
The following property of $G_\beta^\delta$ is straightforward.
\begin{lem}
	The Fourier transform of $G_\beta^\delta$ is
	\begin{equation}\label{G-Fourier}
		\hat{G}_\beta^\delta(\xi)
		=\dfrac{1}{(2\pi \vert \xi \vert)^\beta}
		\big(e^{-(2\pi\vert \xi \vert)^\beta \delta}
		-e^{-(2\pi\vert \xi \vert)^\beta/\delta}\big).
	\end{equation}
	As $\delta\rightarrow0$, for every $\theta\in \dot{H}^{-\frac{\beta}{2}}$,
	\begin{equation}\label{G-Fourier-error}
		\langle\theta,G_\beta^\delta \ast \theta\rangle
		\rightarrow
		\int_{\mathbb{R}^2}(2\pi \vert \xi \vert)^{-\beta}
		\vert \hat{\theta}(\xi) \vert^2\,\mathrm{d}\xi
		=\langle\theta,G_\beta \ast \theta\rangle
		=(2\pi)^{-\beta}\Vert \theta \Vert_{\dot{H}^{-\frac{\beta}{2}}}^2.
	\end{equation}
\end{lem}

Now we give pointwise estimates on $G_\beta^\delta$. The proof of this lemma relies on the estimate of the fractional heat kernel, see Appendix \ref{error-G}.

\begin{lem}\label{errer-kernel}
	For every nonnegative integer $m$, there exists a constant $C_m$ such that for all $x\in \mathbb{R}^2$,
	\begin{equation}\label{0G}
		\vert D^mG_\beta^\delta(x) \vert
		\lesssim_m \dfrac{1}{\vert x \vert^{m+2-\beta}
		}.
	\end{equation}
	Moreover, the following estimates hold:
	\begin{equation}\label{1G}
		\begin{split}
			\vert \nabla(G_\beta-G_\beta^\delta)(x) \vert
			&\lesssim_{\beta}\delta^{1/2}
			{\bf 1}_{\{\delta^{1/(\beta+3)}\leq \vert x \vert\leq \delta^{-1/\beta}\}}   \\
			&\quad+\vert x \vert^{\beta-3}{\bf 1}_{\{\vert x \vert\leq \delta^{1/(\beta+3)}\text{ or } \vert x \vert\geq \delta^{-1/\beta}\}},
		\end{split}
	\end{equation}
	\begin{equation*}\label{2G}
		\vert D^2(G_\beta-G_\beta^\delta)(x) \vert
		\lesssim_{\beta}
		\delta+\vert x \vert^{\beta-4}{\bf1}_{\{\vert x \vert \leq \delta^{1/(4+\beta)}\}}.
	\end{equation*}
	Note that all the implicit constants are independent of $\delta$.
\end{lem}

\subsection{Smoothing the initial data}
Given initial condition $\theta_0\in \dot{H}^{-\frac{\beta}{2}}$,  we claim that there exists a family $(\theta_0^{\delta})_{\delta}$ satisfying
\begin{equation}\label{condition1}
	\begin{split}
		&\theta_0^{\delta}\in L^1\cap L^{\infty}\cap \dot{H}^{-\frac{\beta}{2}},\\
		&\int_{\mathbb{R}^2}\theta_0^{\delta}(x)\, \mathrm{d}x=0,\\
		&\int_{\mathbb{R}^2}\vert x \vert^m\vert \theta_0^{\delta}(x) \vert\, \mathrm{d}x< \infty,\quad \forall m\in \mathbb{N},
	\end{split}
\end{equation}
and
\begin{equation}\label{condition2}
	\theta_0^{\delta}\rightarrow\theta_0 \quad\text{ in } \dot{H}^{-\frac{\beta}{2}}\text{ as } \delta\rightarrow0.
\end{equation}
Precisely, let $v_0=K*\theta_0$, where $K$ is the Biot-Savart kernel, then $v_0\in \dot{H}^{1-\frac{\beta}{2}}$ and $\mathrm{curl}(v_0)=\theta_0$. Since $C_c^{\infty}$ is dense in $\dot{H}^{1-\frac{\beta}{2}}$, we can choose a family $(v_0^{\delta})_\delta$ of $C_c^{\infty}$ vector fields converging to $v_0$ in $\dot{H}^{1-\frac{\beta}{2}}$ and take $\theta_0^{\delta}=\mathrm{curl}(v_0^{\delta})$. Due to the smoothness, compact support of $v_0^{\delta}$ and the definition of $\theta_0^\delta$ as a curl, the properties \eqref{condition1} are satisfied. The convergence \eqref{condition2} follows from the convergence of $v_0^{\delta}$.

If $\theta_0\in L^p\cap \dot{H}^{-\frac{\beta}{2}}~(\text{resp.}~ L^p\cap L^1\cap \dot{H}^{-\frac{\beta}{2}})$ for some $1<p<\infty$, then besides the properties \eqref{condition1} and \eqref{condition2}, we also claim that the approximate sequence $(\theta_0^{\delta})$  satisfies
\begin{equation}\label{condition4}
	\theta_0^{\delta}\rightarrow \theta_0 \quad\text{ in } L^p~(\text{resp.}~L^1\cap L^p)\text{ as } \delta\rightarrow0.
\end{equation}
%In addition, if $\theta_0\in L^p\cap L^1\cap \dot{H}^{-\frac{\beta}{2}}$ for some $1<p<\infty$, then besides the properpties \eqref{condition1}, \eqref{condition2}, we also claim that the approximation sequence $(\theta_0^{\delta})$  satisfies
%\begin{equation}
%	\theta_0^{\delta}\rightarrow \theta_0 \quad\text{in } L^1\cap L^p\text{ as } \delta\rightarrow0.
%\end{equation}
The proof of this claim can be found in \cite[Section 3]{CogMau}.

Without loss of generality, we can always assume that in addition to the conditions \eqref{condition1} and \eqref{condition2}, the family $(\theta_0^\delta)_\delta$ also satisfies the following bounds: for some $0<2\bar{\epsilon}<\min \big\{\frac{2\alpha+\beta-2}{\beta+4},\frac{\beta-1}{\beta+3},\frac{1}{4} \big\}$, as $\delta\rightarrow0$,
\begin{equation}\label{condition3}
	\begin{split}
		\delta\Vert \theta_0^\delta \Vert_{L^1}
		\int_{\mathbb{R}^2} \vert x \vert^2
		\vert \theta_0^\delta(x)\vert\, \mathrm{d}x=o(1),\\
		%\delta^{\alpha/6}\Vert \theta_0^\delta %\Vert_{L^1}
		%\Vert \theta_0^\delta \Vert_{L^\infty}=o(1),\\
		\delta^{\bar{\epsilon}}(\Vert \theta_0^\delta \Vert_{L^\infty}
		+\Vert \theta_0^\delta \Vert_{L^1})=o(1).
		%\\ \delta^{\frac{\beta-1}{\beta+3}}\Vert \theta_0^\delta \Vert_{L^1}^{1/2}\Vert \theta_0^\delta \Vert_{L^\infty}^{1/2}(\Vert \theta_0^\delta \Vert_{L^\infty}+\Vert \theta_0^\delta \Vert_{L^1})=o(1).
	\end{split}
\end{equation}
Indeed, we can make this possible by relabeling the parameter $\delta$ of the family $(\theta_0^\delta)_\delta$.

\section{A priori estimates of the regularized model}\label{sec-a-priori}
In this section, we define the regularized model for mSQG equation \eqref{mSQG} using the regularized objects in the last section and obtain a priori estimates of energy and time continuity, which are essential in the proof of Theorem \ref{exist}, namely the existence of weak solutions.

\subsection{Regularized model}
After defining the regularized objects in the last section, we can consider the regularized equation for \eqref{mSQG}: for all $0<\delta<1$,
\begin{equation}\label{reg-mSQG}
	\left\{
	\begin{lgathered}
		\mathrm{d}\theta^\delta + (K_{\beta}^\delta \ast \theta^\delta)\cdot \nabla \theta^\delta \, \mathrm{d}t + \sum_{k} \sigma_k^\delta\cdot \nabla \theta^\delta \, \mathrm{d}B^k=\dfrac{c_\delta}{2}\Delta \theta^\delta\, \mathrm{d}t, \\
		%u^\delta=K_{\beta}^\delta*\theta^\delta,\quad
		\theta^\delta(0,\cdot) = \theta_0^\delta,
	\end{lgathered}
	\right.
\end{equation}
where $c_\delta$ is the constant in \eqref{eq-c-delta}. This is a Vlasov-type equation with smooth interaction kernel and smooth common noise. The properties of the solution to equation \eqref{reg-mSQG} are listed in the next lemma. For a detailed proof, we refer the reader to  \cite[Appendix C]{CogMau}.

\begin{lem}\label{reg-mSQG-sol}
	For all $\theta_0^\delta$ satisfying \eqref{condition1}, for every filtered probability space $(\Omega,\mathcal{F},(\mathcal{F}_t)_t,\P)$ and every sequence $(B^k)_k$ of independent real Brownian motions, there exists a $\dot{H}^{-\frac{\beta}{2}}$ solution to equation \eqref{reg-mSQG}, that is, an  $(\mathcal{F}_t)_t$-progressively measurable $\dot{H}^{-\frac{\beta}{2}}(\mathbb{R}^2)$-valued process  $\theta^\delta$ such that $\P\text{-a.s.}$, for every $t\in [0,T]$,
	\begin{align*}
		\theta_t^\delta
		=\theta_0^\delta-\int_{0}^{t}\mathrm{div}((K_\beta^\delta \ast \theta_r^\delta)\theta_r^\delta)\, \mathrm{d}r  %\\
		-\sum_{k}\int_{0}^{t}\mathrm{div}(\sigma_k^\delta\theta_r^\delta)\, \mathrm{d}B_r^k
		+\frac{c_\delta}{2}\int_{0}^{t}\Delta\theta_r^\delta\, \mathrm{d}r,
	\end{align*}
	where the equality holds in distribution. The solution $\theta^\delta$ also satisfies that for all $1\leq p\leq \infty$,
	\begin{equation*}
		\begin{split}
			\sup\limits_{t\in [0,T]} \Vert \theta_t^\delta \Vert_{L^p}\leq \Vert \theta_0^\delta \Vert_{L^p},\quad \P\text{-}a.s.
			%&\mathbb{E}\bigg[\sup\limits_{t\in [0,T]} \Vert \theta_t^\delta \Vert_{\dot{H}^{-\frac{\beta}{2}}}^2\bigg] < \infty.
		\end{split}	
	\end{equation*}
	Moreover, the solution is unique in the space $L^{\infty}\big([0,T];L^p(\Omega;L^2)\big)$ for every $p>2$.
\end{lem}

%Note that without loss of generality, we can always assume that in addition to the conditions \eqref{condition1} and \eqref{condition2}, the family $(\theta_0^\delta)_\delta$ also satisfies the following bounds, as $\delta\rightarrow0$:
%\begin{equation}\label{condition3}
%	\begin{split}
%		\Vert \theta_0^\delta \Vert_{L^1}
%		\int_{\mathbb{R}^2} \vert x \vert^2
%		\vert \theta_0^\delta(x)\vert \mathrm{d}x=o(1),\\
%		\delta^{\alpha/6}\Vert \theta_0^\delta \Vert_{L^1}
%		\Vert \theta_0^\delta \Vert_{L^\infty}=o(1),\\
%		\delta^{\frac{\beta-1}{\beta+3}}(\Vert \theta_0^\delta \Vert_{L^\infty}
%		+\Vert \theta_0^\delta \Vert_{L^1})=o(1),
%		\\ \delta^{\frac{\beta-1}{\beta+3}}\Vert \theta_0^\delta \Vert_{L^1}^{1/2}\Vert \theta_0^\delta \Vert_{L^\infty}^{1/2}(\Vert \theta_0^\delta \Vert_{L^\infty}
%		+\Vert \theta_0^\delta \Vert_{L^1})=o(1).
%	\end{split}
%\end{equation}
%Indeed, we can make this possible by relabeling the family $(\theta_0^\delta)_\delta$.

The following two technical lemmas illustrate that the approximation errors are small, which will be used in the passage to the limit $\delta\rightarrow0$.

\begin{lem}\label{approx-energy}
	Assume \eqref{condition1}, \eqref{condition2} and \eqref{condition3} on the initial condition $\theta_0^\delta$. Suppose $\theta^\delta$ is the unique solution to the regularized model \eqref{reg-mSQG}. We have
	\begin{equation*}
		\mathbb{E}\bigg[\sup\limits_{t\in [0,T]}\Big\vert \Vert\theta_t^\delta \Vert_{\dot{H}^{-\frac{\beta}{2}}}^2
		-(2\pi)^\beta \langle \theta_t^\delta,G_\beta^\delta\ast \theta_t^\delta \rangle \Big\vert \bigg]\rightarrow0, \quad \text{ as }\delta\rightarrow0.
	\end{equation*}
\end{lem}
\begin{proof}
	By the definition of $G_\beta^\delta$, we have
	\begin{equation*}
		\begin{split}
			&\Big\vert \Vert\theta_t^\delta \Vert_{\dot{H}^{-\frac{\beta}{2}}}^2
			-(2\pi)^\beta \langle \theta_t^\delta,G_\beta^\delta\ast \theta_t^\delta \rangle \Big\vert  \\
			&=\int_{\mathbb{R}^2} \vert \xi \vert^{-\beta} \big(1-e^{-(2\pi\vert \xi \vert)^\beta \delta} \big)
			\big\vert \widehat{\theta_t^\delta}(\xi) \big\vert^2\,\mathrm{d}\xi
			+\int_{\mathbb{R}^2} \vert \xi \vert^{-\beta}
			e^{-(2\pi\vert \xi \vert)^\beta/\delta}
			\big\vert \widehat{\theta_t^\delta}(\xi) \big\vert^2\,\mathrm{d}\xi \\
			&\leq (2\pi)^\beta\delta\int_{\mathbb{R}^2} \big| \widehat{\theta_t^\delta}(\xi) \big|^2\,\mathrm{d}\xi
			+\int_{\mathbb{R}^2} \vert \xi \vert^{-\beta}
			e^{-(2\pi\vert \xi \vert)^\beta/\delta}
			\big\vert \widehat{\theta_t^\delta}(\xi) \big\vert^2\,\mathrm{d}\xi.
		\end{split}
	\end{equation*}
	Note that $\widehat{\theta_t^\delta}(0)=\int_{\mathbb{R}^2}\theta_t^\delta(x)\, \mathrm{d}x=\int_{\mathbb{R}^2}\theta_0^\delta(x)\, \mathrm{d}x=0$, we get
	\begin{align*}
		\big\vert \widehat{\theta_t^\delta}(\xi) \big\vert
		\leq \big\vert \widehat{\theta_t^\delta}(0) \big\vert
		+ \big\Vert \nabla \widehat{\theta_t^\delta} \big\Vert_{L^{\infty}} \vert \xi \vert %\\
		\leq %\bigg\lvert \int_{\mathbb{R}^2}\theta_t^\delta(x)\, \mathrm{d}x \bigg\rvert+
		2\pi \vert \xi \vert \int_{\mathbb{R}^2}\vert x \vert\, \vert \theta_t^\delta(x) \vert\, \mathrm{d}x.
	\end{align*}
	Hence we have
	\begin{align*}
		\int_{\mathbb{R}^2} \vert \xi \vert^{-\beta}
		e^{-(2\pi\vert \xi \vert)^\beta/\delta}
		\big\vert \widehat{\theta_t^\delta}(\xi) \big\vert^2\,\mathrm{d}\xi
		&\leq
		4\pi^2\int \vert \xi \vert^{2-\beta}e^{-(2\pi\vert \xi \vert)^\beta/\delta} \bigg(\int \vert x \vert\, \vert \theta_t^\delta(x) \vert\, \mathrm{d}x\bigg)^2 \,\mathrm{d}\xi \\
		&\lesssim \delta \bigg(\int \vert x \vert\, \vert \theta_t^\delta(x) \vert\, \mathrm{d}x\bigg)^2.
	\end{align*}
	Then we obtain
	\begin{equation*}
		\left\vert \Vert\theta_t^\delta \Vert_{\dot{H}^{-\frac{\beta}{2}}}^2
		-(2\pi)^\beta \langle \theta_t^\delta,G_\beta^\delta\ast \theta_t^\delta \rangle \right\vert
		\lesssim \delta \bigg( \Vert \theta_t^\delta \Vert_{L^2}^2+\Big(\int \vert x \vert\, \vert \theta_t^\delta(x) \vert\, \mathrm{d}x\Big)^2 \bigg).
	\end{equation*}
	By Lemma \ref{reg-mSQG-sol}, $\Vert \theta_t^\delta \Vert_{L^2}^2
	\leq  \Vert \theta_0^\delta \Vert_{L^2}^2
	\leq  \Vert \theta_0^\delta \Vert_{L^1}
	\Vert \theta_0^\delta \Vert_{L^\infty}$; by \cite[Lemma C.1]{CogMau}, we get
	\begin{align*}
		\mathbb{E}\bigg[\sup\limits_{t\in [0,T]}
		\bigg(\int_{\mathbb{R}^2} \vert x \vert\, \vert \theta_t^\delta(x) \vert\, \mathrm{d}x\bigg)^2 \bigg]
		\lesssim \Vert \theta_0^\delta \Vert_{L^1}
		\int_{\mathbb{R}^2} \vert x \vert^2
		\vert \theta_0^\delta(x)\vert\, \mathrm{d}x %\\
		+(\Vert \theta_0^\delta \Vert_{L^\infty}^2
		+\Vert \theta_0^\delta \Vert_{L^1}^2+1)
		\Vert \theta_0^\delta \Vert_{L^1}^2.
	\end{align*}
	Combining the above estimates yields
	\begin{align*}
		&\mathbb{E}\bigg[\sup\limits_{t\in [0,T]}\left\vert \Vert\theta_t^\delta \Vert_{\dot{H}^{-\frac{\beta}{2}}}^2
		-(2\pi)^\beta \langle \theta_t^\delta,G_\beta^\delta\ast\theta_t^\delta \rangle \right\vert \bigg] \\
		&\leq
		\delta \bigg(\Vert \theta_0^\delta \Vert_{L^1}
		\Vert \theta_0^\delta \Vert_{L^\infty}
		+\Vert \theta_0^\delta \Vert_{L^1}
		\int_{\mathbb{R}^2} \vert x \vert^2
		\vert \theta_0^\delta(x)\vert\, \mathrm{d}x
		+(\Vert \theta_0^\delta \Vert_{L^\infty}^2
		+\Vert \theta_0^\delta \Vert_{L^1}^2+1)
		\Vert \theta_0^\delta \Vert_{L^1}^2\bigg).
	\end{align*}
	Then the assertion is true thanks to the assumption \eqref{condition3} on the initial condition.
\end{proof}

\begin{lem}\label{kernel-term-err}
	Assume \eqref{condition1}, \eqref{condition2} and \eqref{condition3} on the initial condition $\theta_0^\delta$. Suppose $\theta^\delta$ is the unique solution to the regularized model \eqref{reg-mSQG}. We have
	\begin{equation*}
		\sup\limits_{t\in [0,T]}
		\big\Vert \vert \nabla(G_\beta-G_\beta^\delta) \vert \ast \vert \theta_t^\delta \vert \big\Vert_{L^\infty}
		\lesssim \delta^{\frac{\beta-1}{\beta+3}}(\Vert \theta_0^\delta \Vert_{L^\infty}
		+\Vert \theta_0^\delta \Vert_{L^1}).
	\end{equation*}
\end{lem}

\begin{proof}
	Using the bound \eqref{1G}, we have
	\begin{align*}
		\vert \nabla(G_\beta-G_\beta^\delta) \vert \ast \vert \theta_t^\delta \vert (x)
		&=\int_{\mathbb{R}^2}
		\vert \nabla(G_\beta-G_\beta^\delta) (x-y)\vert \, \vert \theta_t^\delta(y) \vert \,\mathrm{d}y \\
		&\lesssim\int_{\vert x-y \vert\leq \delta^{\frac{1}{\beta+3}}}
		\vert \theta_t^\delta(y) \vert \,
		\vert x-y \vert^{\beta-3} \,\mathrm{d}y \\
		&\quad+\int_{\vert x-y \vert\geq \delta^{-\frac{1}{\beta}}}
		\vert \theta_t^\delta(y) \vert \,
		\vert x-y \vert^{\beta-3} \,\mathrm{d}y  \\
		&\quad+\delta^{\frac{1}{2}}\int_{\delta^{\frac{1}{\beta+3}}\leq
			\vert x-y \vert\leq \delta^{-\frac{1}{\beta}}}
		\vert \theta_t^\delta(y) \vert \,\mathrm{d}y \\
		& \lesssim \delta^{\frac{\beta-1}{\beta+3}}
		\Vert \theta_t^{\delta} \Vert_{L^\infty}
		+\delta^{\frac{1}{\beta}} \Vert \theta_t^{\delta} \Vert_{L^1}
		+\delta^{\frac{1}{2}}\Vert \theta_t^{\delta} \Vert_{L^1}.
	\end{align*}
	Hence, by Lamma \ref{reg-mSQG-sol} and condition \eqref{condition3} we complete the proof.
\end{proof}

\subsection{A priori energy estimate}\label{energy}

In this subsection, our goal is to establish the main energy estimate on the solution $\theta^\delta$ to \eqref{reg-mSQG}, namely the expected value of the energy $\Vert \theta^\delta \Vert_{\dot{H}^{-\beta/2}}^2$. The key point is to obtain a control of the $H^{-\beta/2+1-\alpha}$ norm through the special structure of the Kraichnan covariance $Q$.

Throughout this subsection and the next, we always fix a filtered probability space	$(\Omega,\mathcal{F},(\mathcal{F}_t)_t,\P)$ with the usual conditions and a sequence of independent real $(\mathcal{F}_t)$ Brownian motions $(B^k)_k$ on it. We take $Q^\delta,G_\beta^\delta$ and $K_\beta^\delta$ as before, and $\theta^\delta$ to be the solution of the regularized equation \eqref{reg-mSQG} on $(\Omega,\mathcal{F},(\mathcal{F}_t)_t,\P)$ with the Brownian motions $(B^k)_k$.

We start by computing the approximate $\dot{H}^{-\frac{\beta}{2}}$ norm of the solution in the regularized model.

\begin{lem}\label{Ito}
	We have $\P$-a.s. for every $t\in [0,T]$,
	\begin{equation}
		\begin{split}
			&\langle \theta_t^\delta,G_\beta^\delta \ast \theta_t^\delta \rangle
			-\int_{0}^{t} \iint\limits_{\mathbb{R}^2\times \mathbb{R}^2}
			\mathrm{tr}\big[(Q^\delta(0)-Q^\delta(x-y))D^2G_\beta^\delta(x-y)\big]
			\theta_r^\delta(x)\theta_r^\delta(y)\, \mathrm{d}x\mathrm{d}y\mathrm{d}r   \\
			&=\langle \theta_0^\delta,G_\beta^\delta \ast \theta_0^\delta \rangle
			-2\int_{0}^{t}\sum_{k}
			\langle \sigma_k^\delta\cdot\nabla\theta_r^\delta,G_\beta^\delta \ast \theta_r^\delta \rangle\, \mathrm{d}B_r^k.
		\end{split}
	\end{equation}
	In particular, for every $t\in [0,T]$:
	\begin{equation}
		\begin{split}
			&\mathbb{E}[\langle \theta_t^\delta,G_\beta^\delta \ast \theta_t^\delta \rangle]
			-\langle \theta_0^\delta,G_\beta^\delta \ast \theta_0^\delta \rangle \\
            &= \int_{0}^{t}\mathbb{E}\iint\limits_{\mathbb{R}^2\times \mathbb{R}^2}
			\mathrm{tr}\big[(Q^\delta(0)-Q^\delta(x-y))D^2G_\beta^\delta(x-y)\big]
			\theta_r^\delta(x)\theta_r^\delta(y)\, \mathrm{d}x\mathrm{d}y\mathrm{d}r.
		\end{split}
	\end{equation}
\end{lem}

\begin{proof}
	Consider the functional $F:H^{-4}\ni \theta\rightarrow \langle \theta,G_\beta^\delta \ast \theta \rangle\in \mathbb{R}$,
	whose Fr\'echet derivatives are
	\begin{align*}
		DF(\theta)v&=2\langle v,G_\beta^\delta \ast \theta \rangle, \\
		D^2F(\theta)[v,w]&=2\langle v,G_\beta^\delta \ast w \rangle.
	\end{align*}
	The classical It\^o formula on Hilbert space (see \cite[Theorem 4.17]{DaPrato1992}) implies that
	\begin{equation}
		\begin{split}
			\mathrm{d}\langle \theta^\delta,G_\beta^\delta \ast \theta^\delta \rangle
			&=-2\big\langle (K_\beta^\delta \ast \theta^\delta)\cdot\nabla \theta^\delta,
			G_\beta^\delta \ast \theta^\delta \big\rangle \, \mathrm{d}t
			-2\sum_{k} \langle \sigma_k^\delta\cdot\nabla \theta^\delta,
			G_\beta^\delta \ast \theta^\delta \rangle \, \mathrm{d}B^k \\
			&\quad+\langle c_\delta\Delta\theta^\delta,G_\beta^\delta \ast \theta^\delta  \rangle \, \mathrm{d}t
			+\sum_{k}\langle \sigma_k^\delta\cdot\nabla \theta^\delta,G_\beta^\delta \ast (\sigma_k^\delta\cdot\nabla \theta^\delta) \rangle \, \mathrm{d}t.
		\end{split}\label{ItoFormula}
	\end{equation}
	Note that $K_\beta^\delta=\nabla^\perp G_\beta^\delta$ by definition, then by integration by parts, we obtain
	\begin{equation*}
		\langle (K_\beta^\delta \ast \theta^\delta)\cdot\nabla \theta^\delta,
		G_\beta^\delta \ast \theta^\delta \rangle
		=-\langle \theta^\delta,
		\nabla^\perp (G_\beta^\delta \ast \theta^\delta)
		\cdot\nabla (G_\beta^\delta \ast \theta^\delta)\rangle =0.
	\end{equation*}
	The quadratic variation of the martingale term is integrable, indeed,
	\begin{align*}
		\mathbb{E}\sum_{k}\vert \langle \sigma_k^\delta\cdot\nabla \theta^\delta,
		G_\beta^\delta \ast \theta^\delta \rangle \vert^2
		&=\mathbb{E} \iint\limits_{\mathbb{R}^2\times \mathbb{R}^2}
		\theta^\delta(x)\theta^\delta(y)\nabla G_\beta^\delta \ast \theta^\delta(x)\cdot Q^\delta(x-y)
		\nabla G_\beta^\delta \ast \theta^\delta(y) \, \mathrm{d}x\mathrm{d}y  \\
		&\leq\Vert Q^\delta \Vert_{L^\infty} \Vert \theta_0^\delta \Vert_{L^1}^4\Vert \nabla G_\beta^\delta\Vert_{L^{\infty}}^2
		<\infty.
	\end{align*}
	Hence the martingale term vanishes after taking expectation.
	Let us turn to the  last two terms. Notice that $Q^\delta(0)=c_\delta I_2$, then
	\begin{align*}
		\langle c_\delta\Delta\theta^\delta,G_\beta^\delta \ast \theta^\delta  \rangle =
		\langle \theta^\delta,c_\delta\Delta G_\beta^\delta \ast \theta^\delta  \rangle
		=\iint\limits_{\mathbb{R}^2\times \mathbb{R}^2}
		\mathrm{tr}\big[Q^\delta(0)D^2G_\beta^\delta(x-y)\big]
		\theta^\delta(x)\theta^\delta(y) \, \mathrm{d}x\mathrm{d}y,
	\end{align*}
	and
	\begin{align*}
		\sum_{k}\langle \sigma_k^\delta\cdot\nabla \theta^\delta,G_\beta^\delta \ast (\sigma_k^\delta\cdot\nabla \theta^\delta) \rangle
		&= \sum_{k}\sum_{i,j}\langle \partial_i(\sigma_k^{\delta,i}\theta^\delta),
		\partial_jG_\beta^\delta \ast (\sigma_k^{\delta,j}\theta^\delta) \rangle \\
		&=-\sum_{i,j}\sum_{k}\langle \sigma_k^{\delta,i}\theta^\delta,
		\partial_{ij}^2G_\beta^\delta \ast (\sigma_k^{\delta,j}\theta^\delta) \rangle  \\
		&=-\iint\limits_{\mathbb{R}^2\times \mathbb{R}^2}
		\mathrm{tr}\big[Q^\delta D^2G_\beta^\delta\big](x-y)\theta^\delta(x)\theta^\delta(y)\, \mathrm{d}x\mathrm{d}y.
	\end{align*}
	Substituting all above into \eqref{ItoFormula}, we complete the proof.
\end{proof}

Combining Lemmas \ref{approx-energy} and  \ref{Ito} and the condition   \eqref{condition2}, we get the following bound.

\begin{cor}\label{ItoCor}
	We have, for every $t\in [0,T]$,
	\begin{align*}
		&\mathbb{E}\left[\Vert \theta_t^\delta \Vert_{\dot{H}^{-\frac{\beta}{2}}}^2\right]
		-(2\pi)^\beta\int_{0}^{t}\mathbb{E}\iint\limits_{\mathbb{R}^2\times \mathbb{R}^2}
		\mathrm{tr}\big[(Q^{\delta}(0)-Q^\delta) D^2G_\beta^\delta\big](x-y)\theta_r^\delta(x)\theta_r^\delta(y)\, \mathrm{d}x\mathrm{d}y\mathrm{d}r \\
		&=\Vert \theta_0 \Vert_{\dot{H}^{-\frac{\beta}{2}}}^2 + o(1),
	\end{align*}
	where $o(1)$ tends to $0$ as $\delta\rightarrow 0$.
\end{cor}

Due to Corollary \ref{ItoCor}, our task now is to bound the term $\mathrm{tr}\big[(Q^{\delta}(0)-Q^\delta) D^2G_\beta^\delta\big]$. We split this term as follows: for every $x\neq0$,
\begin{equation}
	\begin{split}
		\mathrm{tr}\big[(Q^{\delta}(0)-Q^\delta(x)) D^2G_\beta^\delta(x)\big]
		&=\mathrm{tr}\big[(Q(0)-Q(x))D^2G_\beta(x)\big]\varphi(x)  \\
		&\quad+\mathrm{tr}\big[(Q(0)-Q(x))D^2(G_\beta^\delta-G_\beta)(x)\big]\varphi(x) \\
		&\quad+\mathrm{tr}\big[\big( Q^{\delta}(0)-Q^\delta(x)-(Q(0)-Q(x)) \big) D^2G_\beta^\delta(x)\big]\varphi(x) \\
		&\quad+\mathrm{tr}\big[(Q^{\delta}(0)-Q^\delta(x)) D^2G_\beta^\delta(x)\big](1-\varphi(x)) \\
		&=:A(x)+R_1(x)+R_2(x)+R_3(x),
	\end{split}\label{keyterm}
\end{equation}
where $\varphi$ is a radial smooth function satisfying $0\leq \varphi \leq1 $ everywhere, $\varphi(x)=1$ for $\vert x \vert\leq1$ and $\varphi(x)=0$ for $\vert x \vert\geq2$.

As we will see, $A$ is the key term in \eqref{keyterm}.

\begin{lem}\label{A-bound}
	Suppose $2\alpha+\beta>2$, then for the term $A(x) =\mathrm{tr}\big[(Q(0)-Q(x))D^2G_\beta(x)\big]\varphi(x)$, there exist two positive constants $c,C$ such that
	\begin{equation*}
		\hat{A}(\xi)\leq -c\langle \xi \rangle^{-(2\alpha+\beta-2)}
		+C\langle \xi \rangle^{-\beta}.
	\end{equation*}
\end{lem}

\begin{proof}
	According to the structure of the covariance matrix $Q$ in Lemma \ref{Q-structure}, we get, for every $x\neq 0$,
	\begin{align*}
		&\mathrm{tr}\big[(Q(0)-Q(x))D^2G_\beta(x)\big]  \\
		=&-\frac{2-\beta}{\gamma(\beta)}\frac{1}{\vert x \vert^{4-\beta}}
		\mathrm{tr}\bigg[  \Big( (B_N(0)-B_N(\vert x \vert))I_2
		+(B_N(\vert x \vert)-B_L(\vert x \vert))\frac{x\otimes x}{\vert x \vert^2} \Big)\Big(I_2-(4-\beta)\frac{x\otimes x}{\vert x\vert^2}\Big)\bigg] \\
		=&-\frac{2-\beta}{\gamma(\beta)}\frac{1}{\vert x \vert^{4-\beta}}
		\big[ (3-\beta)B_L(\vert x \vert)-B_N(\vert x \vert)-(2-\beta)B_N(0) \big]   \\
		=&-\frac{2-\beta}{\gamma(\beta)}\frac{1}{\vert x \vert^{4-\beta}}
		\big[ (\beta_N-(3-\beta)\beta_L)\vert x \vert^{2\alpha}
		+\mathrm{Rem}_{u^2}(\vert x \vert)-(3-\beta)\mathrm{Rem}_{1-u^2}(\vert x \vert) \big].
	\end{align*}
	Recalling the fact $\beta_N=(1+2\alpha)\beta_L$ in Lemma \ref{Q-structure} and the definition of Riesz kernel, we have
	\begin{equation}
		\begin{split}
			\mathrm{tr}\big[(Q(0)-Q(x))D^2G_\beta(x)\big]
			=&-\frac{2-\beta}{\gamma(\beta)}(2\alpha+\beta-2)\beta_L
			\gamma(2\alpha+\beta-2)G_{2\alpha+\beta-2}(x)  \\
			\quad&+\frac{2-\beta}{\gamma(\beta)}\frac{1}{\vert x \vert^{4-\beta}}\big((3-\beta)\mathrm{Rem}_{1-u^2}(\vert x \vert) -\mathrm{Rem}_{u^2}(\vert x \vert)\big).
		\end{split}
	\end{equation}
	
	For the control of $\varphi G_{2\alpha+\beta-2}$ in Fourier modes, by \cite[Lemma 4.3]{CogMau}, there exists some constant $C>0$ such that
	\begin{equation*}
		\widehat{\varphi G}_{2\alpha+\beta-2}(\xi)
		\geq \frac{1}{2}(2\pi)^{2-2\alpha-\beta}
		\langle \xi \rangle^{2-2\alpha-\beta}-C\langle \xi \rangle^{-\beta}.
	\end{equation*}
Concerning the remainder terms, taking $2-\beta<\epsilon<(4-2\alpha-\beta)\wedge1$ fixed, by Lemma \ref{A.2} we get
	\begin{equation*}
		\big| \mathcal{F} \big(\vert\cdot\vert^{\beta-4}\mathrm{Rem}_{1-u^2}(\vert\cdot\vert)\varphi \big)(\xi) \big|
		+ \big| \mathcal{F} \big(\vert\cdot\vert^{\beta-4}\mathrm{Rem}_{u^2}(\vert\cdot\vert)\varphi \big)(\xi) \big|
		\lesssim \langle \xi \rangle^{-2+\epsilon}.
	\end{equation*}
	Therefore, by Young's inequality, for $\bar{\delta}$ to be determined later and a constant $C_{\bar{\delta}}>0$,
	\begin{equation*}
		\big| \mathcal{F} \big(\vert\cdot\vert^{\beta-4}\mathrm{Rem}_{1-u^2}(\vert\cdot\vert)\varphi \big)(\xi) \big| + \big| \mathcal{F} \big(\vert\cdot\vert^{\beta-4}\mathrm{Rem}_{u^2}(\vert\cdot\vert)\varphi \big)(\xi) \big|
		\lesssim \bar{\delta}\langle \xi \rangle^{2-2\alpha-\beta}+C_{\bar{\delta}}\langle \xi \rangle^{-\beta}.
	\end{equation*}
	Now choosing $\bar{\delta}=(2\alpha+\beta-2)\frac{(2-\beta)\gamma(2\alpha+\beta-2)}{4(2\pi)^{2\alpha+\beta-2}\gamma(\beta)}\beta_L$, we conclude that for some constant $C>0$,
	\begin{equation*}
		\hat{A}(\xi)\leq -(2\alpha+\beta-2) \frac{(2-\beta)\gamma(2\alpha+\beta-2)\beta_L}{4(2\pi)^{2\alpha+\beta-2}\gamma(\beta)}\langle \xi \rangle^{2-2\alpha-\beta}
		+C\langle \xi \rangle^{-\beta}.
	\end{equation*}
	The proof is complete.
\end{proof}

Now we turn to bound the terms $R_1=\mathrm{tr}\big[(Q(0)-Q(x))D^2(G_\beta^\delta-G_\beta)(x)\big]\varphi(x)$  and $R_2=\mathrm{tr}\big[(Q^{\delta}(0)-Q^\delta(x)-(Q(0)-Q(x)))D^2G_\beta^\delta(x)\big]\varphi(x)$. In fact, both of them are small.

\begin{lem}\label{R1R2-bound}
	For any $0<\epsilon<2\alpha$
	%$0<\epsilon<(2\alpha+\beta-2)\wedge1$
	, we have
	\begin{align*}
		\vert R_1(x)\vert &\lesssim \delta\varphi(x)+\vert x \vert^{2\alpha+\beta-4}{\bf1}_{\{\vert x \vert \leq \delta^{1/(4+\beta)}\}}\varphi(x),  \\
		\vert R_2(x)\vert &\lesssim_\epsilon \delta^{\epsilon}
		\vert x \vert^{2\alpha+\beta-4-\epsilon}\varphi(x).
	\end{align*}
\end{lem}

\begin{proof}
	For the term $R_1$, by Lemmas \ref{Q-structure} and \ref{errer-kernel}, we have
	\begin{align*}
		\vert R_1(x) \vert &\lesssim \vert Q(0)-Q(x) \vert \,
		\vert D^2G_\beta^\delta(x)-D^2G_\beta(x)\vert \varphi(x) \\
		&\lesssim_\beta\vert x \vert^{2\alpha}\big(\delta+\vert x \vert^{\beta-4}{\bf1}_{\{\vert x \vert \leq \delta^{1/(4+\beta)}\}}\big)\varphi(x) \\
		&\lesssim \delta\varphi(x)+\vert x \vert^{2\alpha+\beta-4}{\bf1}_{\{\vert x \vert \leq \delta^{1/(4+\beta)}\}}\varphi(x).
	\end{align*}
	Notice that both of the Fourier transforms of $Q^\delta$ and $Q$  are even functions, so
	\begin{align*}
		Q^{\delta}(0)-Q^\delta(x)-(Q(0)-Q(x))
		&=\int \langle \xi\rangle^{-(2+2\alpha)}\Big(I_2-\dfrac{\xi\otimes \xi}{\lvert \xi\rvert^2 }\Big) (1-e^{2\pi ix\cdot \xi}) \big(\widehat{\rho^\delta}(\xi)^2-1 \big)\, \mathrm{d}\xi \\
		&=\int \langle \xi\rangle^{-(2+2\alpha)}\Big(I_2-\dfrac{\xi\otimes \xi}{\lvert \xi\rvert^2 }\Big)(1-\cos(2\pi x\cdot \xi)) \big(\widehat{\rho^\delta}(\xi)^2-1 \big)\,\mathrm{d}\xi.
	\end{align*}
	For every $a\in \mathbb{R}$ and $0<\epsilon<2\alpha$, it holds
	\begin{equation*}
		\vert 1-\cos(a)\vert \leq\frac{1}{2}a^2\wedge2\leq2a^{2\alpha-\epsilon},
	\end{equation*}
	hence we obtain
	\begin{align*}
		\big\vert (Q^{\delta}(0)-Q^\delta(x))-(Q(0)-Q(x)) \big\vert
		&\lesssim \vert x \vert^{2\alpha-\epsilon}
		\int\langle \xi\rangle^{-(2+2\alpha)}
		\vert \xi \vert^{2\alpha-\epsilon}{\bf1}_{\{\vert \xi \vert\geq1/\delta\}}\, \mathrm{d}\xi  \\
		&\lesssim \vert x \vert^{2\alpha-\epsilon}
		\int_{1/\delta}^{\infty}\rho^{-1-\epsilon} \, d\rho
		=\frac{1}{\epsilon}\vert x \vert^{2\alpha-\epsilon} \delta^{\epsilon}.
	\end{align*}
	With the uniform bound on $D^2G_\beta^\delta$ in Lemma \ref{errer-kernel}, we get
	\begin{equation*}
		\vert R_2(x) \vert \lesssim \delta^{\epsilon}
		\vert x \vert^{2\alpha+\beta-4-\epsilon}\varphi(x).
	\end{equation*}
	Hence the proof is complete.
\end{proof}

We give the bound of $R_3=\mathrm{tr}\big[(Q^{\delta}(0)-Q^\delta(x)) D^2G_\beta^\delta(x)\big](1-\varphi(x))$ in Fourier modes.
\begin{lem}\label{R3-bound}
	For $0<\alpha<1<\beta<2$, we have for all $\xi\in \mathbb{R}^2$,
	\begin{equation*}
		\big\vert \widehat{R_3}(\xi) \big\vert
		\lesssim
		\langle \xi \rangle^{-2-2\alpha}
		%_\epsilon
		%\vert \xi \vert^{-1}{\bf1}_{\{\vert \xi \vert\leq1\}}
		%+\vert \xi \vert^{-2-2\alpha+\epsilon}{\bf1}_{\{\vert \xi \vert\geq1\}}
		\lesssim_{\beta} \vert \xi \vert^{-\beta}.
	\end{equation*}
\end{lem}

%Once we observe that the kernel $G_\beta\,(1<\beta<2) $ decays more rapidly than the Green kernel $G$ and the singularity near the origin has been removed in term $R_3$, it is easy to see that the proof of this lemma is almost identical to \cite[Lemma 27]{CogMau}, where they handle with the case $\beta=2$, that is, the Green kernel $G$.

\begin{proof}
	By the bound \eqref{0G}, for all nonnegative integer $m$, we have
	\begin{equation*}
		\begin{split}
			\big\Vert \vert\cdot\vert^m \mathcal{F} \big[D^2G_\beta^\delta\,(1-\varphi)\big] \big\Vert_{L^\infty}
			&= \big\Vert \mathcal{F} \big[D^{2+m}G_\beta^\delta\,(1-\varphi)\big] \big\Vert_{L^\infty} \\
			&\lesssim \big\Vert D^{2+m}G_\beta^\delta\,(1-\varphi) \big\Vert_{L^1}  \\
			&\lesssim \big\Vert\vert x \vert^{-(m+4-\beta)} {\bf1}_{\{\vert x \vert\geq1\}} \big\Vert_{L^1}
			\lesssim_\beta 1.
		\end{split}
	\end{equation*}
	Then letting $m=0$ and $m=4$, we get the bound
	\begin{equation}\label{Fourier-cutoff}
		\big\vert\mathcal{F} \big[D^2G_\beta^\delta\,(1-\varphi) \big](\xi)\big\vert
		\lesssim \langle \xi \rangle^{-4},\quad \text{for all } \xi\in \mathbb{R}^2.
	\end{equation}
	Now it remains to estimate the Fourier transform of the term $\mathrm{tr} \big[Q^\delta D^2G_\beta^\delta \big] (1-\varphi)$. By \eqref{Q-delta} and \eqref{Fourier-cutoff}, it holds that for all $\xi\in\mathbb{R}^2$,
	\begin{equation*}
		\begin{split}
			\big\vert \mathcal{F} \big(\mathrm{tr}\big[ Q^\delta D^2G_\beta^\delta \big](1-\varphi) \big)(\xi)\big\vert
			&=
			\big\vert \mathrm{tr} \big[\widehat{Q^{\delta}}\ast \mathcal{F} \big(D^2G_\beta^\delta(1-\varphi)\big) \big](\xi) \big\vert \\
			&\leq\int_{\mathbb{R}^2}
			\big\vert \widehat{Q^{\delta}}(\xi-\eta)\big\vert\,
			\big\vert \mathcal{F} \big(D^2G_\beta^\delta(1-\varphi)\big)(\eta) \big\vert \, \mathrm{d}\eta \\
			&\lesssim \int_{\mathbb{R}^2}
			\langle \xi-\eta\rangle^{-2-2\alpha}
			\langle \eta \rangle^{-4}\, \mathrm{d}\eta.
		\end{split}
	\end{equation*}
	For $\vert\eta\vert\leq \vert\xi\vert/2$, we use the bound $\vert \xi-\eta\vert\geq \vert\xi\vert-\vert\eta\vert\geq \vert \xi\vert/2$ and for $\vert\eta\vert\geq \vert\xi\vert/2$, we simply use the bound $\langle \eta \rangle^{-4}\leq \langle \xi \rangle^{-4}$, then we have
	\begin{equation}\label{key}
		\big\vert \mathcal{F} \big(\mathrm{tr} \big[Q^\delta D^2G_\beta^\delta \big](1-\varphi) \big)(\xi) \big\vert
		\lesssim
		\langle \xi \rangle^{-2-2\alpha} +\langle \xi \rangle^{-4}
		\lesssim
		\langle \xi \rangle^{-2-2\alpha}.
	\end{equation}
	Combining the estimates \eqref{Fourier-cutoff} and \eqref{key}, we arrive at
	\begin{equation*}
		\big\vert \widehat{R_3}(\xi) \big\vert
		\lesssim
		\langle \xi \rangle^{-2-2\alpha}
		\leq \vert\xi\vert^{-\beta},
	\end{equation*}
	which completes the proof.
\end{proof}

Now we put together the bound on the key term in Lemma \ref{A-bound} and the bounds on the remainder terms in Lemmas \ref{R1R2-bound} and \ref{R3-bound}.

\begin{lem}\label{extrareg}
	Let $\alpha\in (0,1)$ and $\beta\in(1,2)$ satisfy $2\alpha+\beta>2$. There exist constants $c,C>0$ such that $\P$-a.s., for every $t\in [0,T]$,
	\begin{align*}
		&\iint\limits_{\mathbb{R}^2\times \mathbb{R}^2}
		\mathrm{tr}\big[(Q^\delta(0)-Q^\delta(x-y))D^2G_\beta^\delta(x-y)\big]
		\theta_t^\delta(x)\theta_t^\delta(y)\, \mathrm{d}x\mathrm{d}y \\
		&\leq -c\Vert \theta_t^\delta \Vert_{H^{-\frac{\beta}{2}+1-\alpha}}^2
		+C\Vert \theta_t^\delta \Vert_{\dot{H}^{-\frac{\beta}{2}}}^2+o(1),
	\end{align*}
	where $o(1)$ tends to $0$ as $\delta\rightarrow0$ uniformly on $[0,T]\times \Omega$.
\end{lem}

\begin{proof}
	By the property of Fourier transform, we have
	\begin{align*}
		&\iint\limits_{\mathbb{R}^2\times \mathbb{R}^2}
		\mathrm{tr}\big[(Q^\delta(0)-Q^\delta(x-y))D^2G_\beta^\delta(x-y)\big]
		\theta_t^\delta(x)\theta_t^\delta(y)\, \mathrm{d}x\mathrm{d}y  \\
		&\leq \int_{\mathbb{R}^2} \big(\hat{A}(\xi)+\widehat{R_3}(\xi)\big) \big\vert \widehat{\theta_t^\delta}(\xi)  \big\vert^2\, \mathrm{d}\xi +\iint\limits_{\mathbb{R}^2\times \mathbb{R}^2}
		\big(\vert R_1(x-y) \vert+\vert R_2(x-y) \vert\big) \vert \theta_t^\delta(x)\vert \, \vert \theta_t^\delta(y) \vert \, \mathrm{d}x\mathrm{d}y.
	\end{align*}
	By Lemmas \ref{A-bound} and \ref{R3-bound}, we obtain
	\begin{align*}
		\int_{\mathbb{R}^2}\big(\hat{A}(\xi)+\widehat{R_3}(\xi)\big) \big\vert \widehat{\theta_t^\delta}(\xi) \big\vert^2\, \mathrm{d}\xi
		&\leq -c\int \langle \xi \rangle^{-(2\alpha+\beta-2)}
		\big\vert \widehat{\theta_t^\delta}(\xi) \big\vert^2\, \mathrm{d}\xi
		+C\int \vert \xi \vert^{-\beta} \big\vert \widehat{\theta_t^\delta}(\xi) \big\vert^2\, \mathrm{d}\xi  \\
		&=-c\Vert \theta_t^\delta \Vert_{H^{-\frac{\beta}{2}+1-\alpha}}^2
		+C\Vert \theta_t^\delta \Vert_{\dot{H}^{-\frac{\beta}{2}}}^2.
	\end{align*}
	Take some  $0<2\bar{\epsilon}<\min\big\{ \frac{2\alpha+\beta-2}{\beta+4}, \frac{\beta-1}{\beta+3} \big\}$, by Lemma \ref{R1R2-bound},
	\begin{align*}
		&\iint\limits_{\mathbb{R}^2\times \mathbb{R}^2}
		\big(\vert R_1(x-y) \vert+\vert R_2(x-y) \vert\big) \vert \theta_t^\delta(x)\vert \, \vert \theta_t^\delta(y) \vert\, \mathrm{d}x\mathrm{d}y \\
		&\leq \Vert \theta_0^\delta \Vert_{L^1}\Vert \theta_0^\delta \Vert_{L^\infty}\big(\Vert R_1 \Vert_{L^1}+\Vert R_2 \Vert_{L^1}\big)  \\
		&\lesssim \Vert \theta_0^\delta \Vert_{L^1}\Vert \theta_0^\delta \Vert_{L^\infty}\,\int \Big(
		\delta\varphi(x)+\vert x \vert^{2\alpha+\beta-4}{\bf1}_{\{\vert x \vert \leq \delta^{1/(4+\beta)}\}}+
		\delta^{\bar{\epsilon}}\vert x \vert^{2\alpha+\beta-4-\bar{\epsilon}}\varphi(x)
		\Big)\, \mathrm{d}x
		\\
		&\lesssim \delta^{\bar{\epsilon}} 	\Vert \theta_0^\delta \Vert_{L^1}\Vert \theta_0^\delta \Vert_{L^\infty}
		\lesssim \delta^{\bar{\epsilon}}  \big(	\Vert \theta_0^\delta \Vert_{L^1}+\Vert \theta_0^\delta \Vert_{L^\infty} \big)^2 = o(1),
	\end{align*}
	where the last step is due to \eqref{condition3}. Combining the above two bounds, we get the desired estimate.
\end{proof}

\begin{prop}\label{energy-bound}
	Let $\alpha\in (0,1)$ and $\beta\in(1,2)$ satisfy $2\alpha+\beta>2$, then there exist constants $c,\, C>0$ such that, for every $0<\delta<1$,
	\begin{equation*}
		\sup\limits_{t\in [0,T]} \mathbb{E}\Big[\Vert \theta_t^\delta \Vert_{\dot{H}^{-\frac{\beta}{2}}}^2\Big]
		+c\int_{0}^{T}\mathbb{E}\Big[\Vert \theta_t^\delta \Vert_{H^{-\frac{\beta}{2}+1-\alpha}}^2\Big]\,\mathrm{d}t\leq C\Vert \theta_0 \Vert_{\dot{H}^{-\frac{\beta}{2}}}^2+o(1),
	\end{equation*}
	where $o(1)$ tend to $0$ as $\delta \rightarrow 0$.
\end{prop}
\begin{proof}
	By Corollary \ref{ItoCor} and Lemma \ref{extrareg}, we get, for every $t\in [0,T]$,
	\begin{equation*}
		\mathbb{E}\Big[\Vert \theta_t^\delta \Vert_{\dot{H}^{-\frac{\beta}{2}}}^2\Big]
		+c\int_{0}^{t}\mathbb{E}\Big[\Vert \theta_r^\delta \Vert_{H^{-\frac{\beta}{2}+1-\alpha}}^2\Big]\,\mathrm{d}r
		\leq
		\Vert \theta_0 \Vert_{\dot{H}^{-\frac{\beta}{2}}}^2+ C\int_{0}^{t}
		\mathbb{E}\Big[\Vert \theta_r^\delta \Vert_{\dot{H}^{-\frac{\beta}{2}}}^2\Big]\, \mathrm{d}r +o(1).
	\end{equation*}
	Then the conclusion is clear by Gr\"onwall inequality.
\end{proof}

\subsection{Bounds on time continuity}\label{time}

In this subsection, we establish the a priori bound on time continuity for a solution to \eqref{reg-mSQG}, namely we bound the expected value of $\Vert \theta^\delta \Vert_{C_t^\gamma(\tilde{H}^{-4})}^2$, which is needed in the compactness method to prove the weak existence.

Following the idea of \cite{CogMau}, we introduce a mixed homogeneous-inhomogeneous $\tilde{H}^{-4}$ norm of a scalar tempered distribution $f$, that is, the $H^{-5+\beta}$ norm of the associated velocity field $K_\beta*f$:

\begin{equation*}
	\Vert f \Vert_{\tilde{H}^{-4}}^2
	:=\int_{\mathbb{R}^2}\langle \xi \rangle^{2\beta-10}\vert \xi \vert^{2-2\beta}
	\vert \hat{f}(\xi) \vert^2\, \mathrm{d}\xi =(2\pi)^{2\beta-2} \Vert K_\beta*f \Vert_{H^{-5+\beta}}^2.
\end{equation*}
The space $\tilde{H}^{-4}$ can be identified with the space of divergence-free $H^{-5+\beta}$ vector fields, hence it is a separable Hilbert space.

\begin{lem}\label{time-continuity}
	Let $0<\alpha<1<\beta<2\text{ and }\frac{4}{\beta+1}\leq p\leq 2$. Suppose that
	the initial data $\theta_0\in L^p\cap \dot{H}^{-\frac{\beta}{2}}$ and $(\theta^\delta_0)_\delta$ satisfies the approximation conditions
	\eqref{condition1}--\eqref{condition4}, then for every $0<\gamma<\frac{1}{2}, \lambda\geq 2$, there exists a constant $C=C_{\gamma,\lambda,\theta_0}>0$ such that, for every $0<\delta<1$,
	\begin{equation*}
		\mathbb{E}\Big[ \Vert \theta^\delta \Vert_{C_t^\gamma(\tilde{H}^{-4})}^\lambda \Big] \leq C.
	\end{equation*}
\end{lem}

\begin{proof}
	We estimate the $\tilde{H}^{-4}$ norm of $\theta_t^\delta-\theta_s^\delta$ by using \eqref{reg-mSQG}:
	\begin{align*}
		\mathbb{E}\Big[ \big\Vert \theta_t^\delta-\theta_s^\delta \big\Vert_{\tilde{H}^{-4}}^\lambda \Big]
		&\lesssim \mathbb{E}\bigg[ \Big\Vert \int_{s}^{t}
		(K_\beta^\delta*\theta_r^\delta)\cdot \nabla \theta_r^\delta\, \mathrm{d}r \Big\Vert_{\tilde{H}^{-4}}^\lambda \bigg]
		+\mathbb{E}\bigg[ \Big\Vert \sum_{k}\int_{s}^{t}
		\sigma_k^\delta\cdot \nabla \theta_r^\delta\, \mathrm{d}B_r^k \Big\Vert_{\tilde{H}^{-4}}^\lambda \bigg]\\
		&\quad+\mathbb{E}\bigg[ \Big\Vert \int_{s}^{t}
		c_\delta\Delta \theta_r^\delta\, \mathrm{d}r \Big\Vert_{\tilde{H}^{-4}}^\lambda \bigg] \\
		&:=S_1+ S_2 +S_3.
	\end{align*}
	Calling $u^\delta=K_\beta*\theta^\delta$ and $\tilde{u}^\delta=K_\beta^\delta*\theta^\delta$, we divide the term $S_1$ as
	\begin{align*}
		S_1
		&\lesssim \mathbb{E}\bigg[ \Big\Vert \int_{s}^{t}
		u_r^\delta\cdot \nabla \theta_r^\delta\, \mathrm{d}r \Big\Vert_{\tilde{H}^{-4}}^\lambda \bigg]
		+ \mathbb{E}\bigg[ \Big\Vert \int_{s}^{t}
		(u_r^\delta-\tilde{u}_r^\delta)\cdot \nabla \theta_r^\delta\, \mathrm{d}r \Big\Vert_{\tilde{H}^{-4}}^\lambda \bigg]\\
		&\lesssim (t-s)^{\lambda}\mathbb{E}\bigg[\sup\limits_{r\in [0,1]}\Vert u_r^\delta\theta_r^\delta \Vert_{H^{-2}}^\lambda \bigg]
		+(t-s)^{\lambda}\mathbb{E}\bigg[\sup\limits_{r\in [0,1]}\Vert (u_r^\delta-\tilde{u}_r^\delta)\theta_r^\delta \Vert_{L^2}^\lambda \bigg].
	\end{align*}
	Note that $\theta^{\delta}\in L^p \text{ and }$
	\begin{equation*}
		u^\delta=\nabla^\perp(-\Delta)^{-\frac{\beta}{2}}\theta^\delta=(-\Delta)^{-\frac{\beta-1}{2}}(\nabla^\perp(-\Delta)^{-\frac{1}{2}})
		\,\theta^\delta,
	\end{equation*}
	then by the property of Riesz transform and Hardy-Littlewood-Sobolev inequality (see e.g. \cite[Chapter III, Section 1]{Stein1971} and \cite[Theorem 1.8]{Bahouri2011FourierAA}), we have $u^{\delta}$ is in $L^{p_\ast}$ and $\Vert u^{\delta} \Vert_{L^{p_\ast}}\lesssim \Vert \theta^{\delta} \Vert_{L^p}$, where $p_\ast^{-1}=p^{-1}-\frac{\beta-1}{2}$. Now let $q^{-1}=p^{-1}+p_\ast^{-1}$, by the condition $\frac{4}{\beta+1}\leq p\leq 2$, we know $1\leq q<2$.
	Sobolev embedding implies that $L^{q}\hookrightarrow H^{-2}$, hence together with H\"older inequality, we obtain
	\begin{align*}
		\mathbb{E}\bigg[\sup\limits_{r\in [0,T]}\Vert u_r^\delta\theta_r^\delta \Vert_{H^{-2}}^\lambda \bigg]
		&\lesssim
		\mathbb{E}\bigg[\sup\limits_{r\in [0,T]}\Vert u_r^\delta\theta_r^\delta \Vert_{L^{q}}^\lambda \bigg]
		\lesssim
		\mathbb{E}\bigg[\sup\limits_{r\in [0,T]}\Vert u_r^\delta\Vert_{L^{p_\ast}}^\lambda \Vert\theta_r^\delta \Vert_{L^p}^\lambda \bigg]   \\
		&\lesssim
		\mathbb{E}\bigg[\sup\limits_{r\in [0,T]} \Vert\theta_r^\delta \Vert_{L^p}^{2\lambda} \bigg] \lesssim \Vert\theta_0^\delta \Vert_{L^p}^{2\lambda}
		\lesssim (\Vert\theta_0 \Vert_{L^p}+o(1))^{2\lambda}.
	\end{align*}
	Recall Lemmas \ref{reg-mSQG-sol}, \ref{kernel-term-err} and \eqref{condition3}, we also have
	\begin{equation}\label{A21}
		\begin{split}
			\mathbb{E}\bigg[\sup\limits_{r\in [0,T]}\Vert (u_r^\delta-\tilde{u}_r^\delta)\theta_r^\delta \Vert_{L^2}^\lambda \bigg]
			&\leq
			\mathbb{E}\bigg[\sup\limits_{r\in [0,T]}\Vert u_r^\delta-\tilde{u}_r^\delta\Vert_{L^{\infty}}^\lambda \Vert\theta_r^\delta \Vert_{L^2}^\lambda \bigg]  \\
			&\leq
			\Vert\theta_0^\delta \Vert_{L^2}^\lambda
			\mathbb{E}\bigg[\sup\limits_{r\in [0,T]}\Vert (\nabla G_\beta^\delta-\nabla G_\beta) \ast \theta_r^\delta\Vert_{L^{\infty}}^\lambda \bigg]    \\
			&\lesssim   \Big[\delta^{\frac{\beta-1}{\beta+3}}\Vert \theta_0^\delta \Vert_{L^1}^{1/2}\Vert \theta_0^\delta \Vert_{L^\infty}^{1/2}(\Vert \theta_0^\delta \Vert_{L^\infty}
			+\Vert \theta_0^\delta \Vert_{L^1})\Big]^\lambda \\
			&\lesssim
			\Big[\delta^{\frac{\beta-1}{\beta+3}}(\Vert \theta_0^\delta \Vert_{L^\infty}
			+\Vert \theta_0^\delta \Vert_{L^1})^2\Big]^\lambda
			=o(1).
		\end{split}
	\end{equation}
	Putting all above together, we have the estimate of $S_1$:
	\begin{equation*}
		S_1\lesssim (t-s)^\lambda (\Vert \theta_0 \Vert_{L^p}^{2\lambda}+o(1)).
	\end{equation*}
	
	Applying the Burkholder-Davis-Gundy inequality to the stochastic integral term $S_2$, we have
	\begin{align*}
		S_2&\lesssim\mathbb{E}\bigg[ \bigg(  \int_{s}^{t}\sum_{k} \Vert \mathrm{div} (\sigma_k^\delta\theta_r^\delta) \Vert_{\tilde{H}^{-4}}^2\, \mathrm{d}r\bigg)^{\lambda/2} \bigg]
		\leq \mathbb{E}\bigg[ \bigg(  \int_{s}^{t}\sum_{k} \Vert \sigma_k^\delta\theta_r^\delta \Vert_{H^{-3}}^2 \mathrm{d}r \bigg)^{\lambda/2} \bigg] \\
		&\leq (t-s)^{\lambda/2}\mathbb{E}\bigg[ \sup\limits_{r\in [0,T]}\bigg(\sum_{k}\Vert\sigma_k^\delta\theta_r^\delta \Vert_{H^{-3}}^2\bigg)^{\lambda/2} \bigg].
	\end{align*}
	Exploiting the Fourier transform, the integrand reads as
	\begin{equation}\label{med1}
		\begin{split}
			\sum_{k}\Vert\sigma_k^\delta\theta^\delta \Vert_{H^{-3}}^2
			&=\sum_{k}\int \langle \xi \rangle^{-6}
			\big\vert \widehat{\sigma_k^\delta\theta^\delta}(\xi) \big\vert^2\,\mathrm{d}\xi  \\
			&=\sum_{k} \int \langle \xi \rangle^{-6}
			\iint\sigma_k^\delta(x)\theta^\delta(x)
			e^{-2\pi ix\cdot \xi}\cdot
			\sigma_k^\delta(y)\theta^\delta(y)
			e^{2\pi iy\cdot \xi}\, \mathrm{d}x\mathrm{d}y\mathrm{d}\xi \\
			&=\iint \theta^\delta(x) \theta^\delta(y)\,
			\mathrm{tr}Q^\delta(x-y)\int  \langle \xi \rangle^{-6}
			e^{-2\pi i(x-y)\cdot \xi}\, \mathrm{d}\xi\mathrm{d}x\mathrm{d}y.
		\end{split}
	\end{equation}
	Let  $\psi(x)=\mathrm{tr}\big[Q^\delta(x)\big]\int \langle \xi\rangle^{-6}e^{-2\pi ix\cdot \xi}\, \mathrm{d}\xi$, then
	\begin{equation}\label{med1'}
		\sum_{k}\Vert\sigma_k^\delta\theta^\delta \Vert_{H^{-3}}^2
		=\int \theta^\delta(x)(\psi\ast \theta^\delta)(x)\, \mathrm{d}x\\
		=\int \big\vert \widehat{\theta^\delta}(\xi) \big\vert^2\hat{\psi}(\xi)\, \mathrm{d}\xi.
	\end{equation}
	We directly calculate the Fourier transform of $\psi$: for all $\xi\in \mathbb{R}^2$,
	\begin{equation*}
		\begin{split}
			\hat{\psi}(\xi)
			=\int \mathrm{tr}\widehat{Q^\delta}(\xi-\eta)\langle \eta\rangle^{-6} \,\mathrm{d}\eta
			\leq \int \langle \xi-\eta \rangle^{-2-2\alpha} \langle \eta\rangle^{-6} \,\mathrm{d}\eta.
		\end{split}
	\end{equation*}
	For $\vert \eta\vert\leq \vert \xi \vert/2$, we know $\langle \xi-\eta \rangle^{-2-2\alpha}\leq \langle \xi/2 \rangle^{-2-2\alpha}$ by triangle inequality; for $\vert \eta\vert\geq \vert \xi \vert/2$, we just use the bound $\Vert \langle \cdot \rangle^{-2-2\alpha} \Vert_{L^\infty}\leq1$. Hence we have
	\begin{equation}\label{med3}
		\begin{split}
			\hat{\psi}(\xi)
			&\lesssim
			\int_{\vert \eta\vert\leq \vert \xi \vert/2}
			\langle \xi/2 \rangle^{-2-2\alpha}
			\langle \eta\rangle^{-6} \,\mathrm{d}\eta
			+\int_{\vert \eta\vert\geq \vert \xi \vert/2}
			\langle \eta\rangle^{-6} \,\mathrm{d}\eta  \\
			&\lesssim
			\langle \xi \rangle^{-2-2\alpha}
			\int \langle \eta\rangle^{-6} \,\mathrm{d}\eta
			+ \int_{\vert \eta\vert\geq \vert \xi \vert/2}
			\langle \eta\rangle^{-6} \,\mathrm{d}\eta \\
			&\lesssim \langle \xi \rangle^{-2-2\alpha}
			+\langle \xi \rangle^{-4} \lesssim \langle \xi \rangle^{-2-2\alpha}.
		\end{split}
	\end{equation}
	Therefore we get
	\begin{equation}\label{med2}
		\sum_{k}\Vert\sigma_k^\delta\theta^\delta \Vert_{H^{-3}}^2
		=\int \big\vert \widehat{\theta^\delta}(\xi) \big\vert^2\hat{\psi}(\xi)\, \mathrm{d}\xi
		\lesssim\int \big\vert \widehat{\theta^\delta}(\xi) \big\vert^2 \langle \xi \rangle^{-2-2\alpha}\, \mathrm{d}\xi
		=\Vert \theta^\delta \Vert_{H^{-1-\alpha}}^2
		\lesssim  \Vert \theta^\delta \Vert_{L^p}^2,
	\end{equation}
	the last inequality follows from the Sobolev embedding $L^p\hookrightarrow H^{1-\frac{2}{p}}\hookrightarrow H^{-1-\alpha}$, since $1<p\leq 2$.
	Hence, we obtain
	\begin{equation*}
		S_2\lesssim (t-s)^{\lambda/2}(\Vert \theta_0 \Vert_{L^p}^{\lambda}+o(1)).
	\end{equation*}
	Similarly, for the term $S_3$, we have
	\begin{equation*}
		S_3\leq (t-s)^\lambda\mathbb{E}\bigg[ \sup\limits_{r\in [0,T]}
		\Vert \theta^\delta \Vert_{H^{-2}}^\lambda \bigg]\lesssim (t-s)^{\lambda}(\Vert \theta_0 \Vert_{L^p}^{\lambda}+o(1)).
	\end{equation*}
	Combining the bounds of $S_1$, $S_2$ and $S_3$, we arrive at
	\begin{equation*}
		\mathbb{E}\big[ \Vert \theta_t^\delta-\theta_s^\delta \Vert_{\tilde{H}^{-4}}^\lambda \big]
		\lesssim (t-s)^{\lambda/2}, \quad \text{for all } s,t\in [0,T].
	\end{equation*}
	By the Kolmogorov criterion (see e.g. \cite[Theorem 2.9]{LeGall2016}), we conclude that, for every $0<\gamma<1/2$,
	\begin{equation*}
		\mathbb{E}\Big[ \Vert \theta^\delta \Vert_{C_t^\gamma(\tilde{H}^{-4})}^\lambda \Big] \lesssim_{\beta,p,\lambda,\Vert\theta_0\Vert_{L^p}} 1.
	\end{equation*}
	The proof is complete.
\end{proof}

\begin{rmk}\label{replace}
	It is easy to verify that the equality \eqref{med1'} and the bound \eqref{med2} hold true when we replace $\sigma_k^\delta$ with $\sigma_k$ and $\theta^\delta$ with any $\theta$ in $\dot{H}^{-\frac{\beta}{2}}$.
\end{rmk}

\section{Weak existence}\label{weak-exist}
In this section we prove Theorem \ref{exist},  namely weak existence of \eqref{mSQG}. Combining the a priori bound on energy in Section \ref{energy} and the bound on time continuity in Section \ref{time}, we exploit the classical compactness method, showing tightness of the laws of solutions $(\theta^\delta)_\delta$ to the regularized equation \eqref{reg-mSQG-sol} and showing that any limit is a weak solution.

The key point of the argument is the uniform bound on the $H^{\frac{\beta}{2}-\alpha}$ norm of the velocity field $u^\delta=K_\beta \ast \theta^\delta$. Indeed, for $\beta>2\alpha$, the embedding $H^{\frac{\beta}{2}-\alpha}\hookrightarrow L^2$ is compact on every bounded domain of $\mathbb{R}^2$ and hence the uniform $H^{\frac{\beta}{2}-\alpha}$ bound implies convergence in strong $L_{loc}^2$ topology, which allows to pass to the limit in the nonlinear term.

In order to prove tightness of the laws of the family $(u^\delta)_\delta$, we need to apply the stochastic Aubin-Lions lemma to the triplet of spaces $H^{\frac{\beta}{2}-\alpha}(\mathbb{R}^2)\hookrightarrow
L^2(\mathbb{R}^2)\hookrightarrow H^{-5+\beta}(\mathbb{R}^2)$ and get the compact embedding  $L_t^2(H^{\frac{\beta}{2}-\alpha})\cap C_t^\gamma(H^{-5+\beta})$ into $L_t^2(L^2)$. However, the embedding  $H^{\frac{\beta}{2}-\alpha}(\mathbb{R}^2)\hookrightarrow
L^2(\mathbb{R}^2)$ is not compact. To overcome this technical difficulty, we introduce a suitable weight $w$ in the space $L^2$.

Let $w:\mathbb{R}^2\rightarrow \mathbb{R}$ be a smooth function such that $0<w\leq1$ and
\begin{equation*}
	\sup\limits_{B_R^c}w\rightarrow 0,\quad  \text{as } R\rightarrow \infty,
\end{equation*}
where $B_R=\{x\in \mathbb{R}^2:\vert x \vert<R\}$.
We also assume that the function $w$ decreases at a sufficiently slow rate, for example, we can take $w(x)=(\log \langle x \rangle+1)^{-1}$.
We define the weighted $L^2$ space $L^2(\mathbb{R}^2;w)$ as follows:
%the space of $L_{loc}^2(\mathbb{R}^2)$ functions $f$ satisfying
\begin{equation*}
	L^2(\mathbb{R}^2;w)=\bigg\{ f\in L_{loc}^2(\mathbb{R}^2): \Vert f\Vert_{L^2(\mathbb{R}^2;w)}^2:=\int_{\mathbb{R}^2} \vert f(x) \vert^2w(x)\, \mathrm{d}x<\infty \bigg\}.
\end{equation*}
%\begin{equation*}
%	\Vert f\Vert_{L^2(\mathbb{R}^2;w)}^2:=\int_{\mathbb{R}^2} \vert f(x) \vert^2w(x)\, \mathrm{d}x<\infty.
%\end{equation*}
The following lemma is from \cite[Lemma 6.1]{CogMau}.
\begin{lem}\label{Aubin-Lions}
	Let $0<2\alpha<\beta<2$, the following embedding is compact:
	\begin{equation*}
		L_t^2 \big(H^{\frac{\beta}{2}-\alpha}(\mathbb{R}^2) \big) \cap C_t^\gamma \big(H^{-5+\beta}(\mathbb{R}^2) \big) \subset L_t^2 \big(L^2(\mathbb{R}^2);w \big).
	\end{equation*}
\end{lem}
\begin{rmk}\label{emdedding-space}
	In fact, for every $\epsilon>0$, the space $H^{\frac{\beta}{2}-\alpha}(\mathbb{R}^2)$ compactly embedds into $
	H_{loc}^{\frac{\beta}{2}-\alpha-\epsilon}(\mathbb{R}^2)$. For details, see e.g. \cite[Lemma A.4]{GalLuo23}.
\end{rmk}

Now let us turn to the proof of our existence result Theorem \ref{exist}. As in the previous sections, $Q^\delta, G_\beta^\delta$ and $K_\beta^\delta$, $\theta_0^\delta$ satisfy the conditions in Section \ref{sec-prelim} and $\theta^\delta$ is the solution to the regularized model \eqref{reg-mSQG}.

\subsection*{Step 1: Tightness}
Let $u^\delta=K_\beta \ast \theta^\delta$. We will prove the tightness of the laws of the family $(u^\delta)_\delta$  in the space $L_t^2(L^2(\mathbb{R}^2;w))$. Recall the following facts:
\begin{align*}
	\Vert K_\beta\ast f \Vert_{\dot{H}^{-1+\frac{\beta}{2}}}=(2\pi)^{1-\beta}\Vert f \Vert_{\dot{H}^{-\frac{\beta}{2}}}, \\
	\Vert K_\beta\ast f \Vert_{H^{-5+\beta}}=(2\pi)^{1-\beta}\Vert f \Vert_{\tilde{H}^{-4}}.
\end{align*}
And we have
\begin{equation*}
	\begin{split}
		\Vert K_\beta\ast f\Vert_{\dot{H}^{\frac{\beta}{2}-\alpha}}^2
		&=(2\pi)^{2-2\beta}\int_{\mathbb{R}^2}\vert \xi \vert^{-\beta+2-2\alpha}\vert\hat{f}(\xi) \vert^2 \,\mathrm{d}\xi \\
		&\lesssim \int_{\vert \xi\vert\leq 1}
		\vert \xi \vert^{-\beta}\vert\hat{f}(\xi) \vert^2\, \mathrm{d}\xi
		+\int_{\vert \xi\vert\geq 1} \langle \xi \rangle^{-\beta+2-2\alpha}\vert\hat{f}(\xi) \vert^2 \mathrm{d}\xi \\
		&\lesssim \Vert f\Vert_{\dot{H}^{-\frac{\beta}{2}}}^2+\Vert f\Vert_{H^{-\frac{\beta}{2}+1-\alpha}}^2.
	\end{split}
\end{equation*}
By interpolation and Young's inequalities, we get
\begin{equation*}
	\Vert K_\beta\ast f\Vert_{L^2}\leq 	
	\Vert K_\beta\ast f \Vert_{\dot{H}^{-1+\frac{\beta}{2}}}^{1-\lambda}
	\Vert K_\beta\ast f\Vert_{\dot{H}^{\frac{\beta}{2}-\alpha}}^{\lambda}\leq
	\Vert K_\beta\ast f \Vert_{\dot{H}^{-1+\frac{\beta}{2}}}+
	\Vert K_\beta\ast f\Vert_{\dot{H}^{\frac{\beta}{2}-\alpha}},
\end{equation*}
where $\lambda=\frac{1-\beta/2}{1-\alpha}$. Hence we obtain
\begin{equation*}
	\Vert K_\beta\ast f\Vert_{H^{\frac{\beta}{2}-\alpha}}^2
	=
	\Vert K_\beta\ast f\Vert_{L^2}^2+
	\Vert K_\beta\ast f\Vert_{\dot{H}^{\frac{\beta}{2}-\alpha}}^2
	\lesssim
	\Vert f\Vert_{\dot{H}^{-\frac{\beta}{2}}}^2+\Vert f\Vert_{H^{-\frac{\beta}{2}+1-\alpha}}^2.
\end{equation*}
By Proposition \ref{energy-bound} and Lemma \ref{time-continuity}, we get, for $0<\gamma<1/2$ and some constant $C>0$,
\begin{align}
	&\limsup\limits_{\delta\rightarrow0}\bigg( \sup\limits_{t\in [0,T]}\mathbb{E}\left[ \Vert u_t^\delta \Vert_{\dot{H}^{-1+\frac{\beta}{2}}}^2 \right] + \int_{0}^{T}\mathbb{E}\left[ \Vert u_t^\delta \Vert_{H^{\frac{\beta}{2}-\alpha}}^2 \right]\mathrm{d}t\bigg)
	\leq C\Vert \theta_0 \Vert_{\dot{H}^{-\frac{\beta}{2}}}^2,\label{bd1}\\
	&\limsup\limits_{\delta\rightarrow 0}\mathbb{E}\Big[ \Vert u^\delta \Vert_{C_t^\gamma(H^{-5+\beta})}\Big]\leq C.\label{bd2}
\end{align}
By Lemma \ref{Aubin-Lions}, for $M>0$, the set
\begin{equation*}
	\mathcal{K}_M=\Big\{f\in L_t^2(L^2(\mathbb{R}^2);w): \Vert f \Vert_{L_t^2(H^{\frac{\beta}{2}-\alpha})} + \Vert f \Vert_{C_t^\gamma(H^{-5+\beta})}\leq M \Big\}
\end{equation*}
is compact in $L_t^2(L^2(\mathbb{R}^2);w)$.
We have, by Chebyshev's inequality,
\begin{equation*}
	\begin{split}
		\P(u^\delta \notin \mathcal{K}_M)&\leq
		\P\left(\Vert u^\delta \Vert_{L_t^2(H^{\frac{\beta}{2}-\alpha})}>\frac{M}{2}\right) + \P\left(\Vert u^\delta \Vert_{C_t^\gamma(H^{-5+\beta})}>\frac{M}{2}\right)\\
		&\leq \frac{4}{M^2}\mathbb{E}\left[ \Vert u^\delta \Vert_{\dot{H}^{-1+\frac{\beta}{2}}}^2 \right]
		+ \frac{4}{M^2}\mathbb{E}\Big[ \Vert u^\delta \Vert_{C_t^\gamma(H^{-5+\beta})}\Big].
	\end{split}
\end{equation*}
By the bounds $\eqref{bd1}$ and $\eqref{bd2}$, we can choose $M$ such that the right-hand side is arbitrarily small. Hence the laws of $(u^\delta)_\delta$ are tight in $L_t^2(L^2(\mathbb{R}^2);w)$. Consequently, the laws of $((u^\delta)_\delta,(B^k)_k)$ are tight in $L_t^2(L^2(\mathbb{R}^2);w)\times C_t^\mathbb{N}$ (with $\mathcal{B}(L_t^2(L^2(\mathbb{R}^2);w))\otimes \mathcal{B}(C_t)^\mathbb{N}$ as $\sigma$-algebra).

\subsection*{Step 2: $\P$-a.s. convergence of a subsequence of copies}
By Skorohod representation theorem \cite[Theorem 5.31]{Kallenberg}, there exists a complete probability space $(\tilde{\Omega},\tilde{\mathcal{F}},\tilde{\P})$, a sequence of $L_t^2(L^2(\mathbb{R}^2);w)\times C_t^\mathbb{N}$-valued random variables $(\tilde{u}^{\delta_n},(\tilde{B}^{k,\delta_n})_k)$ on $(\tilde{\Omega},\tilde{\mathcal{F}},\tilde{\P})$, with $\delta_n\rightarrow 0$ as $n\rightarrow \infty$, and a $L_t^2(L^2(\mathbb{R}^2);w)\times C_t^\mathbb{N}$-valued random variable $(\tilde{u},(\tilde{B}^k)_k)$ such that each $(\tilde{u}^{\delta_n},(\tilde{B}^{k,\delta_n})_k)$
has the same law as $(u^{\delta_n},(B^{k})_k)$ and the sequence $(\tilde{u}^{\delta_n},(\tilde{B}^{k,\delta_n})_k)_n$ converges $\tilde{\P}$-a.s. in the topology of $L_t^2(L^2(\mathbb{R}^2);w)\times C_t^\mathbb{N}$ to $(\tilde{u},(\tilde{B}^k)_k)$ as $n\rightarrow \infty$. We call
\begin{equation*}
	\tilde{\theta}^{\delta_n}=(-\Delta)^{-1+\frac{\beta}{2}} (\nabla^{\perp}\cdot\tilde{u}^{\delta_n}),\quad
	\tilde{\theta}=(-\Delta)^{-1+\frac{\beta}{2}}(\nabla^{\perp}\cdot \tilde{u}).
\end{equation*}

We claim that $\tilde{\theta}$ satisfies the bounds \eqref{path-bd} and \eqref{exist-bound}. Since each $\tilde{u}^{\delta_n}$ has the same law as $u^{\delta_n}$, there exists a version of $\tilde{u}^{\delta_n}$ (still called $\tilde{u}^{\delta_n}$) which has H\"older continuous paths with values in $H^{-5+\beta}$ and satisfies the bounds \eqref{bd1} and \eqref{bd2}. Observe that the norms in $C_t^\gamma(H^{-5+\beta})$ and $ L_t^\infty (\dot{H}^{-1+\frac{\beta}{2}})$ are lower semicontinuous functions in $L_t^2(L^2(\mathbb{R}^2);w)$ (see e.g. \cite[Lemma B.2]{CogMau}), the limit $\tilde{u}$ (a version) has trajectories in $L_t^\infty (\dot{H}^{-1+\frac{\beta}{2}})\cap C_t^\gamma(H^{-5+\beta})$.
The uniform $L^p$ bound of trajectories of $\tilde{\theta}$ in \eqref{path-bd}  is  verified as follows.
%by passing to the limit in the bounds in Lemma \ref{reg-mSQG-sol}, using the lower semicontinuity of the norm $L_t^\infty(L^p\cap L^1)$ in the space $L_t^2(L^2(\mathbb{R}^2);w)$.
For every $t\in (0,T),h>0$ sufficiently small and $\varphi\in C_c^\infty$ with $ \Vert \varphi\Vert_{L^{p'}}\leq 1$, we have $\tilde\P$-a.s.,
\begin{align*}
	\frac{1}{2h}\int_{t-h}^{t+h}\langle \tilde{\theta}_s, \varphi\rangle\, \mathrm{d}s
	&=\frac{1}{2h}\int_{t-h}^{t+h}\langle \tilde{\theta}_s-\tilde{\theta}_s^{\delta_n}, \varphi\rangle\, \mathrm{d}s
	+\frac{1}{2h}\int_{t-h}^{t+h}\langle \tilde{\theta}_s^{\delta_n}, \varphi\rangle\, \mathrm{d}s  \\
	&\leq \frac{1}{\sqrt{2h}}\bigg(\int_{0}^{T} \vert \langle \tilde{\theta}_t^{\delta_n}-\tilde{\theta}_t,\varphi \rangle\vert^2\,\mathrm{d}t\bigg)^{1/2}+ \Vert \theta_0^{\delta_n} \Vert_{L^p}.
\end{align*}
Letting $n\rightarrow\infty$, from Lemma \ref{addition} below, we have $\frac{1}{2h}\int_{t-h}^{t+h}\langle \tilde{\theta}_s, \varphi\rangle \mathrm{d}s\leq \Vert \theta_0 \Vert_{L^p}$.  Next let $h\rightarrow0$ and the time continuity implies that $\langle \tilde{\theta}_t,\varphi \rangle \leq \Vert \theta_0\Vert_{L^p}$, $\tilde\P$-a.s. for all $t\in [0,T]$. Then taking supreme at left-hand side for  $\varphi\in C_c^\infty$ with $ \Vert \varphi\Vert_{L^{p'}}\leq 1$, we have $\sup_{t\in [0,T]} \Vert \tilde{\theta}_t \Vert_{L^p} \leq \Vert \theta_0 \Vert_{L^p}$.
%Similarly, we have $\Vert\theta\Vert_{L_t^\infty(L^p)}\leq \Vert \theta_0\Vert_{L^p},\P$-a.s.,
Hence we have the bound \eqref{path-bd}. The proof of the bound \eqref{exist-bound} is very similar, see also \cite[Section 6]{CogMau}.

\begin{comment}
%Finally, since the $L_t^2(H^{\frac{\beta}{2}-\alpha})$ norm and the $L_t^2(\dot{H}^{-1+\frac{\beta}{2}})$ norm are lower semicontinuous in $L_t^2(L^2(\mathbb{R}^2);w)$ (see \cite[Lemma 42]{CogMau}), the limit $\tilde{u}$ also satisfies
\begin{equation*}
\sup\limits_{t\in [0,T]}\mathbb{E}\left[ \Vert \tilde{u}_t \Vert_{\dot{H}^{-1+\frac{\beta}{2}}}^2\right]
+\int_{0}^{T}\mathbb{E}\left[ \Vert \tilde{u}_t \Vert_{H^{\frac{\beta}{2}-\alpha}}^2 \right]\mathrm{d}t
\leq C\Vert \theta_0 \Vert_{\dot{H}^{-\frac{\beta}{2}}}^2,
\end{equation*}
%which shows the bound \eqref{exist-bound}. Indeed, we can pass to the limit in the uniform bound on $\sup_{t\in [0,T]}\mathbb{E}\big[ \Vert \tilde{u} \Vert_{\dot{H}^{-1+\beta/2}}^2\big]$ as follows. First, by lower semicontinuity of the $L_t^2(\dot{H}^{-1+\frac{\beta}{2}})$  norm in $L_t^2(L^2(\mathbb{R}^2);w)$ and the bound \eqref{bd1}, we have
\begin{equation*}
\frac{1}{2h}\int_{t-h}^{t+h}\mathbb{E}\left[ \Vert \tilde{u} \Vert_{\dot{H}^{-1+\frac{\beta}{2}}}^2\right]\mathrm{d}t\leq \liminf_n
\frac{1}{2h}\int_{t-h}^{t+h}\mathbb{E}\left[ \Vert \tilde{u}_s^{\delta_n} \Vert_{\dot{H}^{-1+\frac{\beta}{2}}}^2\right]\mathrm{d}s\leq C\Vert \theta_0 \Vert_{\dot{H}^{-\frac{\beta}{2}}}^2
\end{equation*}
%for every $t\in (0,T)$ and $h>0$ small. Then, we let $h\rightarrow0$ and get $\mathrm{esssup}_{t\in [0,T]}\mathbb{E}\big[ \Vert \tilde{u} \Vert_{\dot{H}^{-1+\beta/2}}^2\big]\leq C\Vert \theta_0 \Vert_{\dot{H}^{-\beta/2}}^2$. The essential supremum becomes a supremum by path continuity of $\tilde{u}$ in $H^{-5+\beta}$.

\end{comment}

Now by classical technique, for each $n$, we can construct a filtration $(\tilde{\mathcal{F}}_t^{\delta_n})_t$ such that $(\tilde{B}^{\delta_n,k})_k$ is a $(\tilde{\mathcal{F}}_t^{\delta_n})_t$ cylindrical Brownian motion and $\tilde{u}^{\delta_n}$ is $(\tilde{\mathcal{F}}_t^{\delta_n})_t$ progressively measurable. The analogous statement holds for $(\tilde{\mathcal{F}}_t)_t$, $(\tilde{B}^k)_k$ and $\tilde{u}$. Moreover, for each $n$, the object $\big(\tilde{\Omega}, \tilde{\mathcal{F}}, (\tilde{\mathcal{F}}_t^{\delta_n})_t, \tilde{\P}, (\tilde{B}^{\delta_n,k})_k, \tilde{\theta}^{\delta_n}\big)$ is a weak solution to regularized model \eqref{reg-mSQG} with $\delta=\delta_n$.

\subsection*{Step 3: Passage to the limit}
For simplicity of notation, we will omit the tildes and write $(\tilde{\theta}^{\delta_n}, \tilde{u}^{\delta_n}, (\tilde{B}^{\delta_n,k})_k)$ as $(\theta^{\delta_n}, u^{\delta_n}, (B^{\delta_n,k})_k)$, and $(\tilde{\theta}, \tilde{u}, (\tilde{B}^k)_k$ as $(\theta, u, (B^k)_k)$. Let $\varphi \in C_c^\infty(\mathbb{R}^2)$, then $\theta^{\delta_n}=(-\Delta)^{-1+\frac{\beta}{2}} (\nabla^{\perp}\cdot u^{\delta_n})$ satisfies the equation
\begin{equation}\label{eqn-step3}
	\begin{split}
		\langle \theta_t^{\delta_n},\varphi \rangle
		&=\langle \theta_0^{\delta_n},\varphi \rangle
		+\int_{0}^{t}\langle (K_\beta^{\delta_n}\ast \theta_r^{\delta_n})\theta_r^{\delta_n},\nabla \varphi\rangle\, \mathrm{d}r
		+\sum_{k}\int_{0}^{t}\langle \sigma_k^{\delta_n}\theta_r^{\delta_n},\nabla\varphi \rangle\, \mathrm{d}B_r^{\delta_n,k}  \\
		&\quad+\frac{c_{\delta_n}}{2}\int_{0}^{t}\langle \theta_r^{\delta_n},\Delta \varphi \rangle\, \mathrm{d}r \\
		&:=A_1+A_2+A_3+A_4.
	\end{split}
\end{equation}
Now we want to let $n\rightarrow \infty$ in each term, possibly along a subsequence, to recover an equation for $\langle \theta,\varphi\rangle$. The following lemma is needed.
\begin{lem}\label{addition}
	Given a function $\varphi \in C_c^\infty(\mathbb{R}^2)$, we have $\P$-a.s.,
	\begin{equation*}
		\int_{0}^{T} \vert \langle \theta_t^{\delta_n}-\theta_t,\varphi \rangle\vert^2\, \mathrm{d}t\rightarrow0, \quad\text{as } n\rightarrow\infty.
	\end{equation*}
\end{lem}

\begin{proof}
	Let $\psi = \nabla^{\perp} (-\Delta)^{-1+\frac{\beta}{2}}\varphi=G_{2-\beta}\ast (\nabla^{\perp}\varphi)$, then we have
	\begin{equation*}
		\psi(x)=c_\beta\int_{\mathbb{R}^2}\frac{\nabla^{\perp}\varphi(y)}{\vert x-y\vert^{\beta}}\, \mathrm{d}y
		=c_\beta\int_{\mathrm{supp}(\varphi)}\frac{\nabla^{\perp}\varphi(y)}{\vert x-y\vert^{\beta}}\, \mathrm{d}y
		\lesssim_{\varphi,\beta} \langle x \rangle^{-\beta},\quad \forall x\in \mathbb{R}^2.
	\end{equation*}
	Hence, by H\"older inequality, we have
	\begin{align*}
		\int_{0}^{T} \vert \langle \theta_t^{\delta_n}-\theta_t,\varphi \rangle\vert^2\, \mathrm{d}t
		=\int_{0}^{T} \vert \langle u_t^{\delta_n}-u_t,\psi \rangle\vert^2\, \mathrm{d}t
		\leq T\int_{0}^{T}\Vert (u_t^{\delta_n}-u_t)\sqrt{w} \Vert_{L^2}^2\, \mathrm{d}t
		\int_{\mathbb{R}^2}\frac{\vert \psi(x)\vert^2}{\vert w(x)\vert}\,\mathrm{d}x.
	\end{align*}
	By the growth condition on $w$ and the above estimate, we know the integral $\int\frac{\vert \psi(x)\vert^2}{\vert w(x)\vert}\, \mathrm{d}x$ is finite, then the conclusion follows from the convergence of $u^{\delta_n}$ to $u$ in the space
	$ L_t^2(L^2(\mathbb{R}^2);w)$.
\end{proof}
By Lemma \ref{addition}, it is obvious that, $\P$-a.s., the left-hand side of \eqref{eqn-step3} converges to $\langle \theta_t,\varphi\rangle$ in $L^2([0,T])$.
The same argument with the fact $c_{\delta_n}\rightarrow \frac{\pi}{2\alpha}$ yields that $\P\text{-a.s.}$,
\begin{equation*}
	A_4\rightarrow \frac{\pi}{4\alpha}\int_{0}^{t}\langle \theta_r,\Delta \varphi\rangle\, \mathrm{d}r \quad\text{in } C([0,T]).
\end{equation*}

The convergence of $A_1$ follows from the condition $\theta_0^{\delta_n}\rightarrow\theta_0$ in $\dot{H}^{-\frac{\beta}{2}}$ as $n\rightarrow\infty$.

Concerning the convergence of the stochastic term $A_3$, we use a classical result (see \cite[Lemma 2.1]{DEBUSSCHE20111123}): if
\begin{equation}\label{quadratic-var}
	\sum_{k}\int_{0}^{T}\big\vert \langle \sigma_k^{\delta_n}\theta_r^{\delta_n},\nabla \varphi-\langle \sigma_k\theta_r,\nabla \varphi\rangle \big\vert^2 \mathrm{d}r \rightarrow0 \quad\text{in probability,}
\end{equation}
then
\begin{equation*}
	\sup\limits_{t\in [0,T]}\bigg\vert \sum_{k}\int_{0}^{t}\langle \sigma_k^{\delta_n}\theta_r^{\delta_n},\nabla\varphi \rangle\, \mathrm{d}B_r^{\delta_n,k}-
	\sum_{k}\int_{0}^{t}\langle \sigma_k\theta_r,\nabla\varphi \rangle\, \mathrm{d}B_r^{k} \bigg\vert\rightarrow0  \quad\text{in probability.}
\end{equation*}
To show \eqref{quadratic-var}, we split as follows:
\begin{align*}
	&\sum_{k}\int_{0}^{T}\big\vert \langle \sigma_k^{\delta_n}\theta_r^{\delta_n},\nabla \varphi-\langle \sigma_k\theta_r,\nabla \varphi\rangle \big\vert^2\, \mathrm{d}r \\
	&\lesssim
	\sum_{k}\int_{0}^{T}\big\vert \langle \sigma_k(\theta_r^{\delta_n}-\theta_r),\nabla \varphi \rangle\big\vert^2\, \mathrm{d}r
	+\sum_{k}\int_{0}^{T}\big\vert \langle (\sigma_k^{\delta_n}-\sigma_k)\theta_r^{\delta_n},\nabla \varphi \rangle \big\vert^2\, \mathrm{d}r \\
	&=:A_{31}+A_{32}.
\end{align*}
For the term $A_{31}$, we split again:
\begin{align*}
	A_{31}&=\sum_{k\leq N}\int_{0}^{T}\big\vert \langle \sigma_k(\theta_r^{\delta_n}-\theta_r),\nabla \varphi \rangle \big\vert^2\, \mathrm{d}r
	+\sum_{k> N}\int_{0}^{T}\big\vert \langle \sigma_k(\theta_r^{\delta_n}-\theta_r),\nabla \varphi \rangle\big\vert^2\, \mathrm{d}r.
\end{align*}
Note that for each $N$ fixed, the first term in the right-hand side tends to $0$ as $n\rightarrow\infty$ by Lemma \ref{addition}.
For each $N\in \mathbb{N}$, take $Q_N(x,y)=\sum_{k>N}\sigma_k(x)\otimes \sigma_k(y)$, it is known that $Q_N$ converges to zero uniformly on each compact set (see e.g. \cite[Lemma 2.3]{GalLuo23}), so we have $\Vert Q_N \Vert_{L^\infty(B_R\times B_R)}\rightarrow0$ as $N\rightarrow\infty$, where $R>0$ such that $\mathrm{supp}(\varphi)\subset B_R$. Then we have $\P\text{-a.s.}$,
\begin{align*}
	&\sum_{k> N}\int_{0}^{T}\vert \langle \sigma_k(\theta_r^{\delta_n}-\theta_r),\nabla \varphi \rangle\vert^2 \mathrm{d}r\\
	&=
	\int_{0}^{T}\iint\limits
	%_{\mathbb{R}^2\times \mathbb{R}^2}
	(\theta_r^{\delta_n}-\theta_r)(x)\nabla\varphi(x)\cdot Q_N(x,y) \nabla\varphi(y) (\theta_r^{\delta_n}-\theta_r)(y)\, \mathrm{d}x\mathrm{d}y \,\mathrm{d}r  \\
	&\leq \Vert Q_N \Vert_{L^\infty(B_R\times B_R)}\int_{0}^{T} \langle \vert\theta_r^{\delta_n}-\theta_r\vert, \vert\nabla\varphi\vert  \rangle^2\, \mathrm{d}r      \\
	&\leq T\Vert Q_N \Vert_{L^\infty(B_R\times B_R)}\big(\Vert \theta_0^{\delta_n}\Vert_{L^p}^2+\Vert\theta_0 \Vert_{L^p}^2\big) \Vert \nabla\varphi \Vert_{L^{p'}}^2  \\
	&\lesssim_{\Vert \theta_0 \Vert_{L^p},T,\varphi} \Vert Q_N \Vert_{L^\infty(B_R\times B_R)}\rightarrow 0,\quad \text{ as } N\rightarrow \infty.
\end{align*}
Thus we get $\P\text{-a.s.}$, $A_{31}\rightarrow0$ .
For the term $A_{32}$, we have
\begin{equation*}
	A_{32}\leq \Vert \varphi \Vert_{H^4}^2 \sum_{k}\int_{0}^{T}\Vert (\sigma_k^{\delta_n}-\sigma_k)\theta_r^{\delta_n} \Vert_{H^{-3}}^2 \,\mathrm{d}r.
\end{equation*}
Recalling Remark \ref{replace} and proceeding as the equalities \eqref{med1} and \eqref{med1'}, replacing $\sigma_k^\delta$ by $\sigma_k^{\delta_n}-\sigma_k$, we get
\begin{align*}
	\sum_{k}\Vert (\sigma_k^{\delta_n}-\sigma_k)\theta_r^{\delta_n} \Vert_{H^{-3}}^2
	=\int \big\vert \widehat{\theta_r^{\delta_n}}(\xi) \big\vert^2\hat{\phi}(\xi)\, \mathrm{d}\xi,
\end{align*}
where
\begin{equation*}
	\phi(x)=\mathrm{tr}\big[Q+Q^{\delta_n}-2Q^{\delta_n,h}\big](x)\int\langle \xi \rangle^{-6}e^{-2\pi ix\cdot \xi} \, \mathrm{d}\xi.
\end{equation*}
For the Fourier transform of $\phi$, proceeding as \eqref{med3} we have, for every $\xi$,
\begin{align*}
	\hat{\phi}(\xi)
	&=\int \mathrm{tr}\big[\hat{Q}+\widehat{Q^{\delta_n}}-2\widehat{Q^{\delta_n,h}}\big](\xi-\eta)\langle \eta \rangle^{-6}\, \mathrm{d}\eta
	\lesssim \int_{\vert \xi-\eta \vert>1/\delta_n}
	\langle \xi-\eta \rangle^{-2-2\alpha}\langle \eta \rangle^{-6}\, \mathrm{d}\eta  \\
	&\leq \delta ^{2\alpha}\int_{\mathbb{R}^2}
	\langle \xi-\eta \rangle^{-2}
	\langle \eta \rangle^{-6}\, \mathrm{d}\eta
	\lesssim \delta ^{2\alpha}\langle \xi \rangle^{-2}.
\end{align*}
%If $\vert \xi \vert\leq 1/(2\delta_n)$, then it can be bounded by
%\begin{equation*}
%	\delta_n^{2+2\alpha} \int_{\vert \xi-\eta\vert\geq 1/\delta_n}\langle \eta \rangle^{-6}d\eta \leq \delta_n^{2+2\alpha}\int_{\vert \eta\vert\geq \vert \xi \vert} \langle \eta \rangle^{-6}d\eta \lesssim \delta_n^{2+2\alpha} \langle \xi \rangle^{-4}.
%\end{equation*}
%If $\vert \xi \vert\geq 1/(2\delta_n)$, then as in \eqref{med1}, it can be bounded by $\langle \xi \rangle^{-2-2\alpha}$. Hence we arrive at
%\begin{equation*}
%		\hat{\phi}(\xi)\lesssim \delta_n^{2+2\alpha}\langle \xi \rangle^{-4}1_{\{\vert \xi \vert\leq 1/(2\delta_n)\}}
%		+ \langle \xi \rangle^{-2-2\alpha}1_{\{\vert \xi \vert\geq 1/(2\delta_n)\}}.
%\end{equation*}
As a consequence,
\begin{align*}
	\sum_{k}\Vert (\sigma_k^{\delta_n}-\sigma_k)\theta_r^{\delta_n} \Vert_{H^{-3}}^2
	\lesssim  \delta_n^{2\alpha}\int \big\vert \widehat{\theta_r^{\delta_n}}(\xi) \big\vert^2 \langle \xi \rangle^{-2}\, d\xi
	%+ \int_{\vert \xi \vert\geq 1/(2\delta_n)}
	%\vert \widehat{\theta_r^{\delta_n}}(\xi) \vert^2
	%\langle \xi \rangle^{-2-2\alpha}\, d\xi \\
	\lesssim \delta_n^{2\alpha} \Vert \theta_r^{\delta_n} \Vert_{\dot{H}^{-\frac{\beta}{2}}}^2.
\end{align*}
Hence we obtain, as $n\rightarrow\infty$,
\begin{equation*}
	\mathbb{E}[A_{32}]\lesssim \Vert \varphi \Vert_{H^4}^2 \delta_n^{2\alpha} \int_{0}^{T}\mathbb{E}\left[\Vert \theta_r^{\delta_n} \Vert_{\dot{H}^{-\frac{\beta}{2}}}^2\right]\mathrm{d}r\rightarrow0.
\end{equation*}
Summarizing the above arguments, we see that \eqref{quadratic-var} holds, so along a subsequence, we have, $\P\text{-a.s.}$,
\begin{equation*}
	A_3\rightarrow \sum_{k}\int_{0}^{t}\langle \sigma_k\theta_r,\nabla\varphi \rangle\, \mathrm{d} B_r^{k}  \quad\text{in } C([0,T]).
\end{equation*}

To cope with the nonlinear term $A_2$, recalling that $u^{\delta_n}=K_\beta\ast \theta^{\delta_n}$, we have
\begin{equation*}
	\begin{split}
		&\int_{0}^{t}\big(\big\langle (K_\beta^{\delta_n}\ast \theta_r^{\delta_n})\theta_r^{\delta_n},\nabla \varphi\big\rangle-
		\big\langle u_r\theta_r, \nabla \varphi  \big\rangle\big)\, \mathrm{d}r    \\
		&=
		\int_{0}^{t}\big\langle ((K_\beta^{\delta_n}-K_\beta)\ast \theta_r^{\delta_n}) \theta_r^{\delta_n}, \nabla\varphi \big\rangle\, \mathrm{d}r
		+\int_{0}^{t}\langle (u_r^{\delta_n}-u_r)\theta_r^{\delta_n}, \nabla \varphi \rangle\, \mathrm{d}r  \\
		&\quad+ \int_{0}^{t}\langle (\theta_r^{\delta_n}-\theta_r)u_r, \nabla \varphi \rangle\, \mathrm{d}r  \\
		&=:A_{21}+A_{22}+A_{23}.
	\end{split}
\end{equation*}
Proceeding as the estimate \eqref{A21}, for every $1\leq \lambda<\infty$, we have $	\mathbb{E}\big[\sup_{t\in [0,T]}A_{21}^\lambda\big]\rightarrow0.$
Hence $A_{21}$ tends to $0$ in $C([0,T])$ $\P$-a.s.. For term $A_{22}$, recalling Remark \ref{emdedding-space}, we have $u^{\delta_n}\rightarrow u$ in $L_t^2(H_{loc}^s)$ for all $0<s<\beta/2-\alpha$. Suppose the support of $\varphi$ is contained in $B_R$, hence by Sobolev embedding with exponent $\frac{1}{q}=\frac{1}{2}-\frac{s}{2}$ and H\"older inequality, we have $\P$-a.s.,
\begin{equation}\label{A22}
	\begin{split}
		\sup\limits_{t\in [0,T]}	A_{22}&\leq
		\Vert u^{\delta_n}-u \Vert_{L_t^2(L^q(B_R))}  \bigg(\int_{0}^{T} \Vert \theta_r^{\delta_n}\nabla\varphi\Vert_{L^{q'}}^2\, \mathrm{d}r\bigg)^{\frac{1}{2}}  \\
		&\leq
		\Vert u^{\delta_n}-u \Vert_{L_t^2(H^s(B_R))}  \sqrt{T}\Vert \nabla\varphi \Vert_{L^\infty} \sup\limits_{t\in [0,T]} \Vert \theta_t^{\delta_n} \Vert_{L^{q'}(B_R)} \\
		&\lesssim \Vert u^{\delta_n}-u \Vert_{L_t^2(H^s(B_R))}
		\sup\limits_{t\in [0,T]} \Vert \theta_t^{\delta_n} \Vert_{L^{p}} \\
		&\lesssim \Vert u^{\delta_n}-u \Vert_{L_t^2(H^s(B_R))}
		\Vert \theta_0^{\delta_n} \Vert_{L^{p}} \rightarrow 0,
	\end{split}
\end{equation}
Note that we have used the condition $p>\frac{2}{1+\beta/2-\alpha}$ to guarantee $p\geq q'$.

Because $\theta^{\delta_n}$ and $u^{\delta_n}$ are related by a nonlocal operator, it is difficult to obtain strong convergence of $\theta^{\delta_n}$ in some space from the strong convergence of $u^{\delta_n}$ in $L_t^2(L^2(\mathbb{R}^2);w)$. As a consequence, we can just obtain convergence of the term $A_{23}$ in some weak sense, which is different from the case in \cite{CogMau}.
Now we prove the term $A_{23}$ converges weakly to zero in the space $L^2(\Omega\times [0,T])$. Consider the map from $L^2(\Omega\times [0,T];L^p)$ to $L^2(\Omega\times [0,T])$ given by
\begin{equation*}
	y(\cdot)\longmapsto \int_{0}^{\cdot}\langle y_r(\cdot), u_r\cdot\nabla\varphi\rangle\, \mathrm{d}r.
\end{equation*}
By the same trick used to deal with the term $A_{22}$, we can show that this map is linear and bounded, hence it is also weakly continuous. It is straightforward to check that $\theta^{\delta_n}\rightharpoonup \theta$ in $L^2(\Omega\times [0,T];L^p)$ up to a subsequence, hence we have
\begin{equation*}
	A_{23}=  \int_{0}^{\cdot}\langle (\theta_r^{\delta_n}-\theta_r)u_r, \nabla \varphi \rangle\, \mathrm{d}r
	\rightharpoonup 0\quad \text{in }  L^2(\Omega\times [0,T]).
\end{equation*}

From the energy bounds, time continuity bounds and the $\P$-a.s. convergence of other terms, we know that all terms in the regularized equation  \eqref{eqn-step3} converge in $L^2(\Omega\times [0,T])$ weakly to the corresponding term for $\theta $.
%, which in turn implies that all terms $\P$-a.s. converge in $L^2([0,T])$.
Thus, for every $X\in L^\infty(\Omega),\phi\in C_c^\infty([0,T])$, we have
\begin{align*}
	&\mathbb{E}\bigg[\int_{0}^{T} X\phi(t)\langle \theta_t,\varphi \rangle\, \mathrm{d}t\bigg]
	=\mathbb{E}\bigg[ \int_{0}^{T}X\phi(t)\langle \theta_0,\varphi \rangle\, \mathrm{d}t \bigg]
	+ \mathbb{E}\bigg[ \int_{0}^{T}X\phi(t)\int_{0}^{t}\langle u_r\theta_r^{\delta_n},\nabla \varphi\rangle\, \mathrm{d}r \mathrm{d}t \bigg]  \\
	&\quad+\mathbb{E}\bigg[ \int_{0}^{T}X\phi(t)\sum_{k}\int_{0}^{t}\langle \sigma_k\theta_r,\nabla\varphi \rangle\, \mathrm{d}B_r^k  \mathrm{d}t \bigg]
	+\frac{\alpha}{4\pi}\mathbb{E}\bigg[ \int_{0}^{T}X\phi(t)	\int_{0}^{t}\langle \theta_r,\Delta \varphi \rangle\, \mathrm{d}r \mathrm{d}t \bigg].
\end{align*}
By the arbitrariness of $X$ and $\phi$, we obtain $\P$-a.s., for a.e. $t\in [0,T]$,
\begin{equation}\label{weak-sol}
	\langle \theta_t,\varphi \rangle=\langle \theta_0,\varphi \rangle
	+\int_{0}^{t}\langle u_r\theta_r,\nabla \varphi\rangle\, \mathrm{d}r
	+\sum_{k}\int_{0}^{t}\langle \sigma_k\theta_r,\nabla\varphi \rangle\, \mathrm{d}B_r^k
	+\frac{\alpha}{4\pi}\int_{0}^{t}\langle \theta_r,\Delta \varphi \rangle\, \mathrm{d}r.
\end{equation}
Since all terms in \eqref{weak-sol} are continuous in time, this equation holds for every $t\in [0,T]$ on a $\P$-null set independent of $t$.
\subsection*{Step 4: Conclusion}
To conclude that $\theta$ is a weak solution to $\eqref{mSQG}$, it is enough to remove the test function in the formulation \eqref{weak-sol}. For each $\varphi \in  C_c^\infty(\mathbb{R}^2)$, equation \eqref{weak-solution} holds tested against $\varphi$ on a $\P$-null set $\Omega_0$ which might depend on $\varphi$.  We can make the $\P$-null set $\Omega_0$ independent of $\varphi$, for $\varphi \in  C_c^\infty(
\mathbb{R}^2)$ in a countable dense set of $H^4(\mathbb{R}^2)$. Then equation \eqref{weak-solution} holds for every $t\in [0,T]$ on the $\P$-null set $\Omega_0$, which completes the proof of Theorem \ref{exist}.

\section{Pathwise uniqueness}\label{sec-uniqueness}
In this section, we prove Theorem \ref{unique}, that is, pathwise uniqueness of $L^1\cap L^p$ solutions to \eqref{mSQG}. The idea of the proof is to estimate the $\dot{H}^{-\frac{\beta}{2}}$ norm of the difference of two solutions $\theta^1$ and $\theta^2$, where the main point is to take advantage of the control of the $H^{-\frac{\beta}{2}+1-\alpha}$ norm of $\theta^1-\theta^2$ induced by the noise to cancel the singularity generated by nonlinear terms. %The uniform $L^1\cap L^p$ bounds on solutions $\theta^1$ and $\theta^2$ is also needed.

The following classical lemma is needed.
\begin{lem}\label{harmonic}
	Let $a+b>0$ and $-1<a,b<1$, we have
	\begin{equation*}
		\Vert fg\Vert_{\dot{H}^{a+b-1}}\lesssim \Vert f\Vert_{\dot{H}^a}\Vert g\Vert_{\dot{H}^b}.
	\end{equation*}
\end{lem}

\begin{proof}[Proof of Theorem \ref{unique}]
	As we know from Remark \ref{condition-unique}, conditions here are stronger than those of Theorem \ref{exist}, so there exists a solution which satisfies
	\eqref{path-bd}--\eqref{exist-bound}. The bound \eqref{path-bd2} follows in a similar manner as \textbf{Step 2} in the proof of Theorem \ref{exist}.
	
	Now let us turn to prove the uniqueness. Let $\theta^1$ and $\theta^2$ be two weak solutions to \eqref{mSQG} on the same filtered probability space $(\Omega,\mathcal{F},(\mathcal{F}_t)_t,\P)$ and with respect to the same sequence $(B^k)_k$ of independent Brownian motions satisfying $\theta^1,\theta^2\in L^\infty(\Omega\times [0,T];L^p\cap L^1)$. The difference $\theta:=\theta^1-\theta^2$ satisfies the following equality in $H^{-4}$:
	\begin{equation*}
		\mathrm{d}\theta+\big[(K_\beta*\theta^1)\cdot \nabla\theta+(K_\beta\ast \theta)\cdot \nabla\theta^2 \big]\, \mathrm{d}t+\sum_{k}\sigma_k\cdot \nabla\theta \, \mathrm{d}B^k=\frac{\pi}{4\alpha}\Delta \theta \, \mathrm{d}t.
	\end{equation*}
	Now applying It\^o formula to $\langle \theta,G_\beta^\delta*\theta \rangle$ and similarly to the computations in Lemma \ref{Ito}, we obtain
	\begin{equation}\label{u-Ito}
		\begin{split}
			\mathrm{d}\langle \theta,G_\beta^\delta\ast \theta \rangle
			&=2\langle \nabla G_\beta^\delta\ast \theta,  (K_\beta\ast \theta^1)\theta\rangle\, \mathrm{d}t +2\langle \nabla G_\beta^\delta\ast \theta,  (K_\beta\ast \theta)\theta^2\rangle\, \mathrm{d}t\\
			&\quad +2\sum_{k} \langle \nabla G_\beta^\delta\ast \theta,\sigma_k\theta \rangle\, \mathrm{d}B^k  \\
			&\quad+\iint\limits_{\mathbb{R}^2\times \mathbb{R}^2}\mathrm{tr}\big[(Q(0)-Q(x-y)) D^2G_\beta^\delta(x-y)\big] \theta(x)\theta(y)\, \mathrm{d}x\mathrm{d}y\mathrm{d}t   \\
			&=:(2I_1+2I_2)\, \mathrm{d}t+2\sum_{k}M_k\, \mathrm{d}B^k+J\,\mathrm{d}t.
		\end{split}
	\end{equation}
	As in the proof of Lemma \ref{Ito}, we know the It\^o integral is a true martingale with zero expectation.
	
	Concerning the term $I_1$, the idea is to control $\theta$ by its $H^{-\frac{\beta}{2}+1-\alpha}$ norm and $\theta^1$ by its $L^p$ norm. We fix $\epsilon>0$ such that $\alpha+\epsilon<\frac\beta2 -\frac1{p\wedge 2}$. We exploit Lemma \ref{harmonic} and get
	\begin{equation*}
		\begin{split}
			\vert I_1 \vert
			&=\vert\, \langle \nabla G_\beta^\delta\ast \theta,  (K_\beta\ast \theta^1)\theta\rangle \, \vert \\
			&\leq \Vert \theta \Vert_{\dot{H}^{-\frac{\beta}{2}+1-\alpha-\epsilon}} \Vert (K_\beta\ast \theta^1)\nabla G_\beta^\delta\ast \theta\Vert_{\dot{H}^{\frac{\beta}{2}-1+\alpha+\epsilon}}  \\
			&\lesssim \Vert \theta \Vert_{\dot{H}^{-\frac{\beta}{2}+1-\alpha-\epsilon}} \Vert K_\beta\ast \theta^1\Vert_{\dot{H}^{2(\alpha+\epsilon)}} \Vert{\nabla G_\beta^\delta\ast \theta}\Vert_{\dot{H}^{\frac{\beta}{2}-\alpha-\epsilon}}  \\
			&\lesssim \Vert \theta \Vert_{\dot{H}^{-\frac{\beta}{2}+1-\alpha-\epsilon}}^2 \Vert \theta^1\Vert_{\dot{H}^{2(\alpha+\epsilon)-\beta+1}}.
		\end{split}
	\end{equation*}
	Taking $\frac{1}{\tilde{p}}=\frac{\beta}{2}-\alpha-\epsilon$ (so $1<\tilde{p}<p\wedge 2$), thanks to the Sobolev embedding  $L^{\tilde{p}}\hookrightarrow \dot{H}^{2(\alpha+\epsilon)-\beta+1}$, we have
	\begin{equation*}
		\vert I_1\vert\lesssim \Vert \theta^1\Vert_{L^{\tilde{p}}}\Vert \theta \Vert_{\dot{H}^{-\frac{\beta}{2}+1-\alpha-\epsilon}}^2.
	\end{equation*}
	By assumption, the $L^p\cap L^1$ norm of $\theta^1$ is uniformly bounded on $\Omega\times [0,T]$, hence there exists a constant $C_1$, such that $(\P\otimes \mathrm{d}t)$-a.s.,
	\begin{equation*}
		\Vert \theta_t^1\Vert_{L^{\tilde{p}}}\leq \Vert \theta_t^1\Vert_{L^{p}}+\Vert \theta_t^1\Vert_{L^{1}}\leq C_1.
	\end{equation*}
	By interpolation, one can check, for $\bar{\epsilon}>0$ to be determined later,
	\begin{equation}\label{interpolation}
		\begin{split}
			\Vert \theta \Vert_{\dot{H}^{-\frac{\beta}{2}+1-\alpha-\epsilon}}^2
			&\leq \Vert \theta\Vert_{\dot{H}^{-\frac{\beta}{2}}}^2+\Vert \theta \Vert_{H^{-\frac{\beta}{2}+1-\alpha-\epsilon}}^2 \\
			&\lesssim \Vert \theta\Vert_{\dot{H}^{-\frac{\beta}{2}}}^2+\bar{\epsilon}\Vert \theta\Vert_{H^{-\frac{\beta}{2}+1-\alpha}}^2+C_{\bar{\epsilon}}\Vert \theta\Vert_{H^{-\frac{\beta}{2}}}^2.
		\end{split}
	\end{equation}
	Taking expectation and integrating in time, we get
	\begin{equation}\label{I1}
		\int_{0}^{t}\mathbb{E}[\vert I_1\vert]\, \mathrm{d}r
		\lesssim \bar{\epsilon}\int_{0}^{t}\mathbb{E}\left[\Vert\theta_r\Vert_{H^{-\frac{\beta}{2}+1-\alpha}}^2\right]\mathrm{d}r
		+(C_{\bar{\epsilon}}+1) \int_{0}^{t}\mathbb{E}\left[\Vert\theta_r\Vert_{\dot{H}^{-\frac{\beta}{2}}}^2\right]\mathrm{d}r.
	\end{equation}
	
	Now we turn to estimate the term $I_2$, note that
	\begin{equation*}
		\langle \nabla G_\beta\ast \theta,(K_\beta\ast \theta)\theta^2 \rangle=\langle \nabla (G_\beta\ast \theta), \nabla^\perp (G_{\beta}\ast \theta)\,\theta^2 \rangle=0,
	\end{equation*}
	so it is enough to control the remainder term with $G_\beta-G_\beta^\delta$. We have, by Lemma \ref{harmonic},
	\begin{equation*}
		\begin{split}
			\vert I_2\vert
			&=\vert\langle \nabla (G_\beta^\delta-G_\beta)\ast \theta,(K_\beta*\theta)\theta^2 \rangle \vert  \\
			&\leq \Vert \theta^2 \Vert_{\dot{H}^{-\frac{\beta}{2}+1-\alpha-\epsilon}} 	
			\Vert (K_\beta*\theta)\nabla (G_\beta^\delta-G_\beta)\ast \theta\Vert_{\dot{H}^{\frac{\beta}{2}-1 +\alpha+\epsilon}} \\
			&\lesssim \Vert \theta^2 \Vert_{\dot{H}^{-\frac{\beta}{2}+1-\alpha-\epsilon}} 	
			\Vert K_\beta\ast \theta\Vert_{\dot{H}^{2(\alpha+\epsilon)}}
			\Vert\nabla (G_\beta^\delta-G_\beta)\ast \theta\Vert_{\dot{H}^{\frac{\beta}{2}-\alpha-\epsilon}} \\
			&\leq \Vert \theta^2 \Vert_{\dot{H}^{-\frac{\beta}{2}+1-\alpha-\epsilon}}
			\Vert\theta\Vert_{\dot{H}^{2(\alpha+\epsilon)-\beta+1}}
			\Vert\nabla (G_\beta^\delta-G_\beta)\ast \theta\Vert_{\dot{H}^{\frac{\beta}{2}-\alpha-\epsilon}}.
		\end{split}
	\end{equation*}
	Still by the Sobolev embedding  $L^{\tilde{p}}\hookrightarrow \dot{H}^{2(\alpha+\epsilon)-\beta+1}$ and the uniform bound of $L^1\cap L^p$ norm of $\theta^1$ and $\theta^2$, we get
	\begin{equation*}
		\begin{split}
			\vert I_2\vert
			&\lesssim \big(\Vert \theta^1\Vert_{L^{\tilde{p}}}+\Vert \theta^2\Vert_{L^{\tilde{p}}}\big)\Vert \theta^2 \Vert_{\dot{H}^{-\frac{\beta}{2}+1-\alpha-\epsilon}} \Vert\nabla (G_\beta^\delta-G_\beta)\ast \theta\Vert_{\dot{H}^{\frac{\beta}{2}-\alpha-\epsilon}}\\
			&\lesssim (C_1+C_2)\Vert \theta^2 \Vert_{\dot{H}^{-\frac{\beta}{2}+1-\alpha-\epsilon}} \Vert\nabla (G_\beta^\delta-G_\beta)\ast \theta\Vert_{\dot{H}^{\frac{\beta}{2}-\alpha-\epsilon}}.
		\end{split}
	\end{equation*}
	Taking expectation and integrating in time, we get
	\begin{align*}
		\int_{0}^{T}\mathbb{E}[\vert I_2\vert]\, \mathrm{d}r
		\lesssim \bigg(\int_{0}^{T}\mathbb{E}\Big[\Vert \theta_r^2\Vert_{\dot{H}^{-\frac{\beta}{2} +1-\alpha-\epsilon}}^2 \Big] \mathrm{d}r\bigg)^{1/2}
		\bigg(\int_{0}^{T}\mathbb{E}\Big[\Vert\nabla (G_\beta^\delta-G_\beta)\ast \theta\Vert_{\dot{H}^{\frac{\beta}{2} -\alpha-\epsilon}}^2\Big]\mathrm{d}r\bigg)^{1/2}.
	\end{align*}
	%Thanks again to the interpolation, we arrive at
	%\begin{align*}
	%	\int_{0}^{T}\mathbb{E}[\vert I_2\vert]
	%	%&
	%	\lesssim
	%	\bigg(\int_{0}^{T}\mathbb{E}\Big[\Vert \theta_r^2\Vert_{\dot{H}^{-\frac{\beta}{2}}}^2+\Vert \theta_r^2\Vert_{H^{-\frac{\beta}{2}+1-\alpha}}^2\Big]\mathrm{d}r\bigg)^{1/2} %\\
	%	%&\cdot
	%	\bigg(\int_{0}^{T}\mathbb{E}\Big[\Vert\nabla (G_\beta^\delta-G_\beta)*\theta\Vert_{\dot{H}^{\frac{\beta}{2}-\alpha-\epsilon}}^2\Big]\mathrm{d}r\bigg)^{1/2}.
	%\end{align*}
	By \eqref{G-Fourier}, we can write the last term as
	\begin{align*}
		&\int_{0}^{T}\mathbb{E}\Big[\Vert\nabla(G_\beta^\delta-G_\beta)*\theta\Vert_{\dot{H}^{\frac{\beta}{2}-\alpha-\epsilon}}^2\Big]\mathrm{d}r \\
		&=\int \vert \xi\vert^{-\beta+2-2(\alpha+\epsilon)}
		\big(1-e^{-(2\pi\vert \xi \vert)^\beta \delta}
		+e^{-(2\pi\vert \xi \vert)^\beta/\delta}\big)^2
		\int_{0}^{T}\mathbb{E}\big[\vert\widehat{\theta}_r(\xi)\vert^2\big] \,\mathrm{d}r\mathrm{d}\xi.
	\end{align*}
	Notice that by \eqref{interpolation},  the integral
	\begin{equation*}
		\int \int_{0}^{T}\vert \xi\vert^{-\beta+2-2(\alpha+\epsilon)}
		\mathbb{E}\big[\vert\widehat{\theta}_r(\xi)\vert^2\big]\, \mathrm{d}r\mathrm{d}\xi=\int_{0}^{T}\mathbb{E}\left[\Vert \theta_r^2\Vert_{\dot{H}^{-\frac{\beta}{2}+1-\alpha-\epsilon}}^2\right]\mathrm{d}r
	\end{equation*}
	is finite, hence by dominated convergence theorem,
	\begin{equation}\label{I2}
		\int_{0}^{T}\mathbb{E}[\vert I_2\vert]\, \mathrm{d}r =o(1)\quad \text{as } \delta\rightarrow0.
	\end{equation}
	
	As in Lemma \ref{extrareg}, the term $J$ in \eqref{u-Ito} provides a control of the $H^{-\frac{\beta}{2}+1-\alpha}$ norm of the difference $\theta$. We divide the quantity in $J$ as follows:
	\begin{equation*}
		\begin{split}
			\mathrm{tr}\big[(Q(0)-Q(x))D^2G_\beta^\delta(x)\big]
			&=\mathrm{tr}\big[(Q(0)-Q(x))D^2G_\beta(x)\big]\varphi(x)  \\
			&\quad+\mathrm{tr}\big[(Q(0)-Q(x))D^2(G_\beta^\delta-G_\beta)(x)\big]\varphi(x) \\
			&\quad+\mathrm{tr}\big[(Q(0)-Q(x)) D^2G_\beta^\delta(x)\big](1-\varphi(x)) \\
			&=:A(x)+R_1(x)+R_3(x),
		\end{split}
	\end{equation*}
	where $\varphi$ is a radial smooth function satisfying $0\leq \varphi \leq1 $ everywhere, $\varphi(x)=1$ for $\vert x \vert\leq1$ and $\varphi(x)=0$ for $\vert x \vert\geq2$. Now we proceed as in Lemmas \ref{A-bound}--\ref{R3-bound} and obtain
	\begin{align*}
		J&\leq \int_{\mathbb{R}^2}(\hat{A}(\xi)+\widehat{R_3}(\xi))\, \vert \widehat{\theta}_t(\xi) \vert^2\, \mathrm{d}\xi
		+\iint\limits_{\mathbb{R}^2\times \mathbb{R}^2}
		\vert R_1(x-y) \vert\, \vert \theta_t(x) \vert\, \vert \theta_t(y) \vert\, \mathrm{d}x\mathrm{d}y \\
		&=-c\int \langle \xi \rangle^{-(2\alpha+\beta-2)}
		\vert \widehat{\theta}_t(\xi) \vert^2\, \mathrm{d}\xi
		+C\int \vert \xi \vert^{-\beta} \vert \widehat{\theta_t}(\xi) \vert^2\, \mathrm{d}\xi \\
		&\quad+C\iint\big(\delta+\vert x-y \vert^{2\alpha+\beta-4}{\bf1}_{\{\vert x-y \vert \leq \delta^{1/(4+\beta)}\}}\big)\varphi(x-y)\,\vert\theta_t(x)\vert\, \vert\theta_t(y)\vert\, \mathrm{d}x\mathrm{d}y.
	\end{align*}
	By H\"older's and Young's inequalities for the last term, we get, taking $\frac{1}{r}=2-\frac{2}{p\wedge2}$,
	\begin{align*}
		J&\leq-c\Vert \theta_t \Vert_{H^{-\frac{\beta}{2}+1-\alpha}}^2
		+C\Vert \theta_t \Vert_{\dot{H}^{-\frac{\beta}{2}}}^2 +C\delta\Vert\theta_t\Vert_{L^1}^2  %\\
		+C\bigg(  \int \vert x \vert^{(2\alpha+\beta-4)r}{\bf1}_{\{\vert x \vert \leq \delta^{1/(4+\beta)}\}}\mathrm{d}x\bigg)^{1/r} \Vert\theta_t\Vert_{L^{p\wedge2}}^2 \\
		&\leq -c\Vert \theta_t \Vert_{H^{-\frac{\beta}{2}+1-\alpha}}^2
		+C\Vert \theta_t \Vert_{\dot{H}^{-\frac{\beta}{2}}}^2 +C\delta\Vert\theta_t\Vert_{L^1}^2+o(1)\Vert\theta_t\Vert_{L^{p\wedge2}}^2,
	\end{align*}
	where we have used the assumption $\alpha>\frac{2}{p}-\frac{\beta}{2}$ to guarantee $(2\alpha+\beta-4)r>-2$. Taking expectation and integrating in time, we get
	\begin{equation}\label{J}
		\int_{0}^{t}\mathbb{E}[J]\, \mathrm{d}r\leq -c\int_{0}^{t}\mathbb{E}\left[\Vert\theta_r\Vert_{H^{-\frac{\beta}{2}+1-\alpha}}^2\right]\mathrm{d}r
		+C \int_{0}^{t}\mathbb{E}\left[\Vert\theta_r\Vert_{\dot{H}^{-\frac{\beta}{2}}}^2\right]\mathrm{d}r+o(1),
	\end{equation}
	where $o(1)$ is uniform for $t\in [0,T]$.
	
	Now integrating over $\Omega\times [0,t]$ in \eqref{u-Ito}, using the bounds \eqref{I1}--\eqref{J} where $\bar{\epsilon}>0$ in \eqref{I1} is chosen small enough and letting $\delta\rightarrow 0$, by $\eqref{G-Fourier-error}$ and the embedding $L^1\cap L^p\subset{\dot{H}^{-\beta/2}}$ in Remark \ref{condition-unique}, we arrive at
	
	\begin{equation*}
		\mathbb{E}\left[\Vert\theta_t\Vert_{\dot{H}^{-\frac{\beta}{2}}}^2\right]\leq \Vert\theta_0\Vert_{\dot{H}^{-\frac{\beta}{2}}}^2
		-\frac{c}{2}\int_{0}^{t}\mathbb{E}\left[\Vert\theta_r\Vert_{H^{-\frac{\beta}{2}+1-\alpha}}^2\right]\mathrm{d}r
		+C \int_{0}^{t}\mathbb{E}\left[\Vert\theta_r\Vert_{\dot{H}^{-\frac{\beta}{2}}}^2\right]\mathrm{d}r.
	\end{equation*}
	Then by Gr\"onwall inequality and $\theta_0=\theta_0^1-\theta_0^2=0$, we get
	\begin{equation*}
		\sup\limits_{t\in [0,T]} \mathbb{E}\left[\Vert \theta_t \Vert_{\dot{H}^{-\frac{\beta}{2}}}^2\right]
		+\frac{c}{2}\int_{0}^{T}\mathbb{E}\left[\Vert \theta_t \Vert_{H^{-\frac{\beta}{2}+1-\alpha}}^2\right]\, \mathrm{d}t\leq \Vert \theta_0 \Vert_{\dot{H}^{-\frac{\beta}{2}}}^2e^{CT}=0,
	\end{equation*}
	which completes the proof.
\end{proof}
		
\appendix

\section{Proof of Lemma \ref{errer-kernel} }\label{error-G}

In this section, we give a proof of Lemma \ref{errer-kernel}, more precisely, a generalized version in $\mathbb{R}^d$. Before we start, we need to do some preparation work about the fractional heat kernel. Throughout this section, we suppose $p(t,x)$ is the fractional heat kernel on $\mathbb{R}^d$ for some integer $d \geq 2$, that is to say,
\begin{equation*}
	\left\{
	\begin{lgathered}
		\partial_t p + (-\Delta)^{\frac{\beta}{2}}p=0
		\quad\text{in }(0,\infty)\times \mathbb{R}^d,\\
		p(0,\cdot)=\delta_0.
	\end{lgathered}
	\right.
\end{equation*}
For $(t,x)\in (0,\infty)\times \mathbb{R}^d$, define $q(t,x)=t(t^{2/\beta}+\vert x \vert^2)^{-\frac{d+\beta}{2}}$.
The following bound of the kernel $p(t,x)$ %crucial to estimate the error between $G_\beta$ and $G_\beta^\delta$.
is well known, see for example \cite{Bogdan2007EstimatesOH, Kwasnicki+2017+7+51}.
\begin{lem}
	\label{0pq}
	There exists a constant $C=C(d,\beta)$ such that for all $(t,x)\in (0,\infty)\times \mathbb{R}^d$,
	\begin{equation}
		C^{-1}q(t,x)%\dfrac{t}{(t^{\frac{2}{\beta}}+\vert x \vert^2)^{\frac{d+\beta}{2}}}
		\leq p(t,x) \leq
		Cq(t,x).%\dfrac{t}{(t^{\frac{2}{\beta}}+\vert x \vert^2)^{\frac{d+\beta}{2}}}.
	\end{equation}
\end{lem}

\begin{rmk}\label{equal-q}
	It is easy to check that
	\begin{equation*}
		\dfrac{t}{(t^{\frac{2}{\beta}}+\vert x \vert^2)^{\frac{d+\beta}{2}}}
		\asymp
		\dfrac{t}{\vert x \vert^{d+\beta}}\wedge t^{-\frac{d}\beta}.
	\end{equation*}
	For simplicity, both sides of the above expression are represented by $q(t,x)$.
\end{rmk}

From a probabilistic point of view, the fractional Laplacian $-(-\Delta)^{\frac{\beta}{2}}$ is the infinitesimal generator of the $\beta$-stable process, which is a pure jump Markov process. That is to say, $p(t,x)$ is the transition probability of a $\beta$-stable process $Z=(Z_t)_{t\geq 0}$. Then by the Bochner subordination formula (see e.g.
\cite[Proposition 8.6]{2004Bertoinbook}), we have
\begin{equation*}
	Z=\sqrt{2}B_S\quad \text{in law},
\end{equation*}
where $B$ and $S$ are independent, $B=(B_t)_{t\geq 0}$ is a standand $d$ dimensional Brownian motion and $S=(S_t)_{t\geq 0}$ is a subordinator with Laplace exponent $\lambda^{\frac{\beta}{2}}$, namely, $\mathbb{E}[e^{-\lambda S_t}]=e^{-\lambda^{\frac{\beta}{2}}t}$ for all $\lambda,t\geq 0$.
Hence by conditioning we get
\begin{equation}
	p(t,x)=\int_{0}^{\infty}h(s,x)\rho_t(s)\, \mathrm{d}s,\label{Bochner}
\end{equation}
where $h(s,x)=(4\pi s)^{-\frac{d}{2}}e^{-\frac{\vert x \vert^2}{4s}}$ is the heat kernel in $\mathbb{R}^d$, $\rho_t$ is the probability density function of $S_t$. Accroding to \cite[Lemma 5]{Bogdan2007EstimatesOH}, there exists a constant $c$ such that $\rho_1(s)\leq c\, s^{-1-\frac{\beta}{2}}$
%e^{-s^{\frac{\beta}{2}}}$,
for all $s>0$, so by scaling property $\rho_t(s)=t^{-\frac{2}{\beta}}\rho_1(t^{-\frac{2}{\beta}}s)$, we have
\[\rho_t(s)\leq c\,t \, s^{-1-\frac{\beta}{2}},\quad t>0.\]
Combining the above estimate and \eqref{Bochner}, the dominated convergence theorem implies that
\begin{equation}
	\label{1p}
	\begin{split}
		\nabla_x p(t,x)&=\int_{0}^{\infty}\nabla_x h(s,x)\rho_t(s)\,\mathrm{d}s
		=-\dfrac{x}{2}\int_{0}^{\infty}
		\dfrac{h(s,x)}{s}\rho_t(s)\,\mathrm{d}s  \\
		&=-2\pi p_{(d+2)}(t,x)\, x,
	\end{split}    	
\end{equation}
where $p_{(d+2)}$ is the heat kernel in dimension $d+2$. Strictly speaking, $p_{(d+2)}(t,x)$ should be $p_{(d+2)}(t,\bar{x})$, where $\bar{x}\in\mathbb{R}^{d+2}$ satisfies $\vert \bar{x} \vert=\vert x \vert$.
Similarly, we obtain that
\begin{equation}\label{2p}
	D_x^2p(t,x)=-2\pi p_{(d+2)}(t,x)\,I_d+4\pi^2p_{(d+4)}(t,x)\, x\otimes x,
\end{equation}
where $p_{(d+4)}(t,x)$ is understood in a similar way and $I_d$ is $d\times d$ unit matrix.
%Strictly speaking, $p_{(d+2)}(t,x)(\,\text {resp. }p_{(d+4)}(t,x))$ should be $p_{(d+4)}(t,\bar{x})$, where $\bar{x}\in\mathbb{R}^{d+2}(\,\text{resp. }\mathbb{R}^{d+4})$ satisfies $\vert \bar{x} \vert=\vert x \vert$.
A simple calculation shows that
\begin{equation}\label{1q2q}
	\begin{split}
		\nabla_xq(t,x)&=-(d+\beta)q_{(d+2)}(t,x)\, x,\\
		D_x^2q(t,x)&= -(d+\beta)q_{(d+2)}(t,x)\, I_d+(d+\beta)(d+2+\beta)q_{(d+4)}(t,x)\, x\otimes x.
	\end{split}
\end{equation}
Combining Lemma \ref{0pq} and \eqref{1p} -- \eqref{1q2q}, we have the following estimates.

\begin{lem}\label{012pq}
	Let $p,q$ be defined as above, we have that for all $(t,x)\in (0,\infty)\times \mathbb{R}^d$,
	\[p(t,x)\asymp q(t,x),\quad \nabla_xp(t,x)\asymp \nabla_xq(t,x),\quad D_x^2p(t,x)\asymp D_x^2q(t,x),\]
	where the constants behind $\asymp$ are only dependent on $\beta$ and $d$.
\end{lem}

\begin{rmk}\label{mpq}
	In fact, the same proof works for the general case $m\in \mathbb{N}$,
	\begin{equation*}
		D_x^mp(t,x)\asymp D_x^mq(t,x).
	\end{equation*}
\end{rmk}

%\section{Proof of Lemma \ref{errer-kernel}}\label{errer-G}
Now we turn to prove Lemma \ref{errer-kernel}. Note that the relation \eqref{G} holds true in dimension $d$ and $G_\beta^\delta$ is defined as in \eqref{G-approx}.
We give a general version of Lemma \ref{errer-kernel} for all dimension $d$.

\begin{lem}
	For every nonnegative integer $m$, there exists a constant $C_m$ such that for all $x\in \mathbb{R}^d$,
	\begin{equation*}
		\vert D^mG_\beta^\delta(x) \vert
		\lesssim_m \dfrac{1}{\vert x \vert^{m+d-\beta}
		}.
	\end{equation*}
	Moreover, the following estimates hold:
	\begin{equation}\label{1Gd}
		\begin{split}
			\vert \nabla(G_\beta-G_\beta^\delta)(x) \vert
			&\lesssim_{\beta,d}\delta^{1/2}
			{\bf 1}_{\{\delta^{1/(\beta+d+1)}\leq \vert x \vert\leq\delta^{-1/\beta}\}}   \\
			&\quad+\vert x \vert^{\beta-d-1}{\bf 1}_{\{\vert x \vert\leq \delta^{1/(\beta+d+1)}\text{ or } \vert x \vert\geq \delta^{-1/\beta}\}},
		\end{split}
	\end{equation}
	\begin{equation}\label{2Gd}
		\vert D^2(G_\beta-G_\beta^\delta)(x) \vert
		\lesssim_{\beta,d}
		\delta+\vert x \vert^{\beta-d-2}{\bf1}_{\{\vert x \vert \leq \delta^{1/(d+2+\beta)}\}}.
	\end{equation}
	Note that all implicit constants are independent of $\delta$.
\end{lem}

\begin{proof}
	We first prove the uniform bound on $\vert D^mG_\beta^\delta(x) \vert$. By Remark \ref{mpq}, we get
	\begin{equation*}
		\vert D^mG_\beta^\delta(x) \vert
		\leq \int_{\delta}^{1/\delta}
		\vert D_x^mp(t,x) \vert\, \mathrm{d}t
		\asymp \int_{\delta}^{1/\delta}
		\vert D_x^mq(t,x) \vert\, \mathrm{d}t
		\lesssim_m \dfrac{1}{\vert x \vert^{m+d-\beta}}.
	\end{equation*}
	Now we move to the bounds on $\nabla(G_\beta-G_\beta^\delta)$ and $D^2(G_\beta-G_\beta^\delta)$. Using \eqref{1p}, \eqref{2p} and Lemma \ref{012pq}, we get
	\begin{equation}\label{B.10}
		\begin{split}
			\vert \nabla(G_\beta-G_\beta^\delta)(x) \vert
			&\leq
			\bigg(\int_{0}^{\delta}+\int_{1/\delta}^{\infty}\bigg)
			\vert \nabla_xp(t,x) \vert\, \mathrm{d}t  \\
			&\asymp
			\bigg(\int_{0}^{\delta}+\int_{1/\delta}^{\infty}\bigg)
			p_{(d+2)}(t,x)\vert x\vert\, \mathrm{d}t\\
			&\asymp
			\bigg(\int_{0}^{\delta}+\int_{1/\delta}^{\infty}\bigg)
			q_{(d+2)}(t,x)\vert x\vert\, \mathrm{d}t.
		\end{split}
	\end{equation}
	Similarly,
	\begin{equation}\label{B.11}
		\vert D^2(G_\beta-G_\beta^\delta)(x) \vert
		\lesssim \bigg(\int_{0}^{\delta}+\int_{1/\delta}^{\infty}\bigg)
		\big(q_{(d+2)}(t,x)+q_{(d+4)}(t,x)\big)\vert x \vert^2\, \mathrm{d}t.
	\end{equation}
	With Remark \ref{equal-q}, a tedious calculation shows that
	\begin{equation}\label{B.12}
		\bigg(\int_{0}^{\delta}+\int_{1/\delta}^{\infty}\bigg)
		q(t,x)\,\mathrm{d}t\leq
		\begin{cases}
			\big(\frac{1}{2}+\frac{\beta}{d-\beta}\big)\vert x \vert^{\beta-d},& \vert x \vert\leq \delta^{1/\beta}\text{ or }\vert x \vert\geq \delta^{-1/\beta};  \\
			\frac{1}{2}\delta^2\vert x \vert ^{-(\beta+d)}+\frac{\beta}{d-\beta}\delta^{\frac{d-\beta}{\beta}},& \delta^{1/\beta}\leq \vert x \vert \leq \delta^{-1/\beta}.     \\
			%\big(\frac{1}{2}+\frac{\beta}{d-\beta}\big)\vert x \vert^{\beta-d},& \vert x \vert\geq \delta^{-1/\beta}.
		\end{cases}
	\end{equation}
	Substituting \eqref{B.12} into \eqref{B.10} and \eqref{B.11} yields
	\begin{equation*}
		\vert \nabla(G_\beta-G_\beta^\delta)(x)\vert\lesssim
		\begin{cases}
			\big(\frac{1}{2}+c\big)\vert x \vert^{\beta-d-1},& \vert x \vert\leq \delta^{1/\beta}\text{ or }\vert x \vert\geq \delta^{-1/\beta};  \\
			\frac{1}{2}\delta^2\vert x \vert ^{-(\beta+d+1)}+c\delta^{\frac{d+2}{\beta}-1}
			\vert x \vert,& \delta^{1/\beta}\leq \vert x \vert \leq \delta^{-1/\beta},     %\\
			%\big(\frac{1}{2}+c\big)\vert x \vert^{\beta-d-1},& \vert x \vert\geq \delta^{-1/\beta},
		\end{cases}
	\end{equation*}
	\begin{equation*}
		\vert D^2(G_\beta-G_\beta^\delta)(x) \vert
		\lesssim
		\begin{cases}
			\big(1+c+\tilde{c}\big)\vert x \vert^{\beta-d-2},& \vert x \vert\leq \delta^{1/\beta}\text{ or }\vert x \vert\geq \delta^{-1/\beta};  \\
			\delta^2\vert x \vert ^{-(\beta+d+2)}+c\delta^{\frac{d+2}{\beta}-1}
			\\\quad+\tilde{c}\delta^{\frac{d+4}{\beta}-1}\vert x \vert^2,& \delta^{1/\beta}\leq \vert x \vert \leq \delta^{-1/\beta},    %\\
			%\big(1+c+\tilde{c}\big)\vert x \vert^{\beta-d-2},& \vert x \vert\geq \delta^{-1/\beta},
		\end{cases}
	\end{equation*}
	where the constants $c=\frac{\beta}{d+2-\beta},\, \tilde{c}=\frac{\beta}{d+4-\beta}$. Simplifying the above estimates, we can derive the assertions of the lemma. Indeed, for $\delta^{1/(\beta+d+1)}\leq\vert x \vert\leq\delta^{-1/\beta}$, we have
	\begin{equation*}
		\begin{split}
			\vert \nabla(G_\beta-G_\beta^\delta)(x) \vert
			\lesssim \delta^2\vert x \vert ^{-(\beta+d+1)}+\delta^{\frac{d+2}{\beta}-1}\vert x \vert
			\lesssim \delta+\delta^{\frac{d+1}{\beta}-1}\lesssim\delta^{1/2};
		\end{split}
	\end{equation*}
	for $\delta^{1/\beta}\leq\vert x \vert\leq\delta^{1/(\beta+d+1)}$, we have
	\begin{equation*}
		\begin{split}
			\vert \nabla(G_\beta-G_\beta^\delta)(x) \vert
			\lesssim \delta^2\vert x \vert ^{-(\beta+d+1)}+\delta^{\frac{d+2}{\beta}-1}\vert x \vert
			\lesssim
			\vert x \vert ^{\beta-d-1}+\vert x\vert
			\lesssim \vert x \vert ^{\beta-d-1}.
		\end{split}
	\end{equation*}
	Thus we have shown \eqref{1Gd}
	\begin{align*}
		\vert \nabla(G_\beta-G_\beta^\delta)(x) \vert
		&\lesssim_{\beta,d}\delta^{1/2}
		{\bf1}_{\{\delta^{1/(\beta+d+1)}\leq \vert x \vert\leq \delta^{-1/\beta}\}}   \\
		&\quad+\vert x \vert^{\beta-d-1}{\bf1}_{\{\vert x \vert\leq\delta^{1/(\beta+d+1)}\text{ or } \vert x \vert\geq \delta^{-1/\beta}\}}.
	\end{align*}
	%\begin{equation*}
	%	\vert D^2(G_\beta-G_\beta^\delta)(x) \vert
	%	\lesssim_{\beta,d}
	%	\delta+\vert x \vert^{\beta-d-2}{\bf1}_{\{\vert x \vert \leq \delta^{1/(d+2+\beta)}\}}.
	%\end{equation*}
	The estimate \eqref{2Gd} can be obtained by the same trick.
\end{proof}

\section{Estimating the remainders}

In this section, we give an estimate of the remainder terms appearing in Lemma \ref{A-bound}. For this purpose, we need the following lemma.
\begin{lem}
	Let $1<\beta<2$ be fixed, we have the following bounds:
	\begin{align*}
		\big|\, \vert x \vert^{\beta-4}\mathrm{Rem}_f(\vert x \vert)\,\big| &\lesssim \vert x \vert^{\beta-2},
		\\
		\big\vert \nabla\big[\,\vert x \vert^{\beta-4}\mathrm{Rem}_f(\vert x \vert)\big]\,  \big\vert &\lesssim \vert x \vert^{\beta-3},
		\\
		\big\vert D^2\big[\,\vert x \vert^{\beta-4}\mathrm{Rem}_f(\vert x \vert)\big]\, \big\vert &\lesssim \vert x \vert^{\beta-4}.
	\end{align*}
\end{lem}
The proofs are almost identical to \cite[Lemma A.1]{CogMau}, so we omit the details here.

\begin{lem}\label{A.2}
	Let $1<\beta<2$ be fixed. Suppose $g:\mathbb{R}^2\rightarrow\mathbb{R}$ is a Borel measurable function, which has support in $\bar{B}_2(0)$ and is $C^2$ on $\mathbb{R}^2\backslash \{0\}$, and assume that for all $x\in \mathbb{R}^2\backslash \{0\}$,
	\begin{equation*}
		\vert g(x) \vert\lesssim\vert x \vert^{\beta-2},\quad
		\vert \nabla g(x)  \vert \lesssim \vert x \vert^{\beta-3},\quad
		\vert D^2g(x)  \vert \lesssim \vert x \vert^{\beta-4}.
	\end{equation*}
	Then for all $2-\beta<\epsilon<1$, we have
	\begin{equation*}
		\vert \hat{g}(\xi) \vert \lesssim_\epsilon \langle \xi \rangle^{-2+\epsilon},
		\quad  \text{for all } \xi \in \mathbb{R}^2.
	\end{equation*}
\end{lem}

\begin{proof}
	First we claim that the $W^{1+\gamma,p}$ norm of $g$ is finite for $0<\gamma<\beta-1$ and $1<p<\frac{2}{3-\beta+\gamma}<2$. Here the $W^{1+\gamma,p}$ norm is defined as
	\begin{equation*}
		\Vert g \Vert_{W^{1+\gamma,p}}^p= \Vert g \Vert_{L^p}^p
		+ \Vert \nabla g \Vert_{L^p}^p
		+\iint\limits_{\mathbb{R}^2\times \mathbb{R}^2}
		\dfrac{\vert\nabla g(x)-\nabla g(y)\vert^p}{\vert x-y \vert^{2+\gamma p}}\,\mathrm{d}x\mathrm{d}y.
	\end{equation*}
	Obviously the $L^p$ norms of $g$ and $\nabla g$ are finite. For the last term, on the one hand we exploit the bound on $\nabla g$ and get,
	\begin{equation*}
		\dfrac{\vert\nabla g(x)-\nabla g(y)\vert^p}{\vert x-y \vert^{2+\gamma p}}
		\lesssim\dfrac{1}{(\vert x \vert \wedge \vert y \vert)^{(3-\beta)p}
			\,\vert x-y \vert^{2+\gamma p}}=:I_1.
	\end{equation*}
	On the other hand, take a $C^1$-path $\eta: [0,1]\rightarrow \mathbb{R}^2$ with $\eta(0)=y,\, \eta(1)=x, \min\limits_{s\in [0,1]}\vert \eta(s) \vert \gtrsim \vert x \vert \wedge \vert y \vert$ and $\max\limits_{s\in [0,1]}\vert \eta'(s) \vert \lesssim \vert x-y \vert$; then exploiting the bound on $D^2g$, we get
	\begin{align*}
		\dfrac{\vert\nabla g(x)-\nabla g(y)\vert^p}{\vert x-y \vert^{2+\gamma p}}
		&\leq \bigg(\int_{0}^{1}
		\vert D^2g(\eta(s)) \vert\,
		\vert \eta'(s) \vert \,\mathrm{d}s\bigg)^p  \frac{1}{\vert x-y \vert^{2+\gamma p}} \\
		&\lesssim \dfrac{\max\limits_{s\in [0,1]}\vert \eta'(s) \vert^p}{\min\limits_{s\in [0,1]}\vert \eta(s) \vert^{(4-\beta)p}}
		\cdot \frac{1}{\vert x-y \vert^{2+\gamma p}} \\
		&\lesssim
		\dfrac{1}{(\vert x \vert \wedge \vert y \vert)^{(4-\beta)p}
			\,\vert x-y \vert^{2-(1-\gamma) p}}=:I_2.
	\end{align*}
	Interpolating between the above two estimates, we get, for all $0<\lambda<1$,
	\begin{align*}
		\dfrac{\vert\nabla g(x)-\nabla g(y)\vert^p}{\vert x-y \vert^{2+\gamma p}} \lesssim
		I_1^\lambda I_2^{1-\lambda}
		%\bigg( \dfrac{1}{(\vert x \vert \wedge \vert y \vert)^{(3-\beta)p}
		%\,\vert x-y \vert^{2+\gamma p}} \bigg)^\lambda
		%\bigg(\dfrac{1}{(\vert x \vert \wedge \vert y \vert)^{(4-\beta)p}
		%\,\vert x-y \vert^{2-(1-\gamma) p}} \bigg)^{1-\lambda}     \\
		= \frac{1}{\vert x-y \vert^{2+(\gamma+\lambda-1) p}}
		\bigg( \dfrac{1}{\vert x \vert^{p(4-\beta-\lambda)}}
		+\dfrac{1}{\vert y \vert^{p(4-\beta-\lambda)}} \bigg).
	\end{align*}
	Since $1<p<\frac{2}{3-\beta+\gamma}<2$, we can choose $0<\lambda<1 -\gamma$ such that $2+(\gamma+\lambda-1)p<2$ and $p(4-\beta-\lambda)<2$, making  the right-hand side above integrable. Therefore, $g$ has finite  $W^{1+\gamma,p}$ norm for $0<\gamma<\beta-1$ and $1<p<\frac{2}{3-\beta+\gamma}<2$.
	Then we have (see \cite[Subsection 2.5]{Triebel1983})
	\begin{equation*}
		\Vert g \Vert_{H^{1+\gamma,p}}:=\Vert \mathcal{F}^{-1} ( \langle  \cdot \rangle^{1+\gamma}\hat{g} ) \Vert_{L^p}
		\lesssim
		\Vert g \Vert_{W^{1+\gamma,p}} <\infty.
	\end{equation*}
	Thus by Hausdorff-Young inequality, we have
	\begin{equation*}
		\langle  \cdot \rangle^{1+\gamma}\hat{g}\in L^{p'},\quad \text{where }  \frac{1}{p}+\frac{1}{p'}=1.
	\end{equation*}
	
	Now we will show $\langle \cdot \rangle^{1+\gamma}\hat{g}$ is actually in $W^{1,p'}$, then by Sobolev embedding, $\langle \cdot \rangle^{1+\gamma}\hat{g}$ is bounded, which completes the proof with $\epsilon=1-\gamma$.
	%To conclude, we would like to show $\langle \cdot \rangle^{1+\gamma}\hat{g}$ is actually bounded, and then take $\epsilon=1-\gamma$. To do so, we will use the compact support of $g$ and Sobolev embedding.
	Note that $g$ has compact support, so $\hat{g}$ is smooth and
	\begin{equation*}
		\vert \nabla (\langle \xi \rangle^{1+\gamma}\hat{g}(\xi)) \vert
		\lesssim \langle \xi \rangle^{1+\gamma} \vert \nabla\hat{g}(\xi) \vert
		+ \langle \xi \rangle^{\gamma} \vert \hat{g}(\xi) \vert
		\lesssim \langle \xi \rangle^{2} \vert \nabla\hat{g}(\xi) \vert
		+ \langle \xi \rangle \vert \hat{g}(\xi) \vert.
	\end{equation*}
	For the first term, observe that $xg(x)$ and $\Delta(xg(x))$ are in $L^p$, hence
	\begin{equation*}
		\langle \xi \rangle^{2} \vert \nabla\hat{g}(\xi) \vert\asymp
		\big\vert   \mathcal{F}\big[(I-\Delta)(xg(x)) \big] (\xi)  \big\vert
		\in L^{p'}.
	\end{equation*}
	Similarly, we have
	\begin{equation*}
		\langle \xi \rangle \vert \hat{g}(\xi) \vert \leq \vert \hat{g}(\xi) \vert + \vert \xi\, \hat{g}(\xi) \vert = \vert \hat{g}(\xi) \vert+ \vert \mathcal{F}(\nabla g)(\xi) \vert\in L^{p'}.
	\end{equation*}
	Consequently, we have proved that for $0<\gamma<\beta-1$ and $1<p<\frac{2}{3-\beta+\gamma}<2$, $\langle \cdot \rangle^{1+\gamma}\hat{g}$ is in $W^{1,p'}$, so by the Sobolev embedding on $\mathbb{R}^2$, $\langle \cdot \rangle^{1+\gamma}\hat{g}$ is bounded.
\end{proof}

\bigskip
		
\noindent \textbf{Acknowledgements.}
The second author is grateful to the financial supports of the National Key R\&D Program of China (No. 2020YFA0712700), the National Natural Science Foundation of China (Nos. 11931004, 12090010, 12090014), and the Youth Innovation Promotion Association, CAS (Y2021002).

		\phantomsection
		%\addcontentsline{toc}{section}{Reference}
		%\bibliography{reference}
	\end{sloppypar}

\end{document}